\documentclass[12pt]{amsart}

\usepackage{graphicx}
\usepackage{fancyhdr}
\usepackage{color}
\usepackage{hyperref}
\usepackage{amsmath, amsthm, amssymb, amsfonts}
\usepackage[mathscr]{eucal}
\usepackage[margin=2.9cm]{geometry}
\usepackage{verbatim}
\usepackage[titletoc]{appendix}
\newtheorem{theorem}{Theorem}[section]
\newtheorem{lemma}[theorem]{Lemma}

\theoremstyle{definition}

\newtheorem{remark}[]{\textbf{Remark}}

\numberwithin{equation}{section}
\begin{document}
\title{ SUBCONVEXITY FOR $GL(3)\times GL(2)$ TWISTS } 
\author{Prahlad Sharma\vspace{2mm}\\\MakeLowercase{with an appendix by} Will Sawin}
\address{School of Mathematics, Tata Institute of Fundamental Research, Mumbai}
\email{prahlad@math.tifr.res.in}
\address{Department of Mathematics, Columbia University, 2990 Broadway, New York, NY 10027, USA}
\email{sawin@math.columbia.edu}
\maketitle
\begin{abstract} Let $\pi$ be a $SL(3,\mathbb{Z})$ Hecke-Maass cusp form, $f$  be a $SL(2,\mathbb{Z})$ holomorphic cusp form or Maass cusp form  and $\chi$ be any non-trivial character $\bmod \, p$, where $p$ is prime. We show that the $L$-function associated with this triplet satisfy
\begin{equation*}
L\left(\frac{1}{2},\pi\times f\times\chi\right)\ll_{\pi,f,\epsilon}
p^{\frac{3}{2}-\frac{1}{16}+\epsilon}.
\end{equation*}
The method also yields the subconvex bound
\begin{equation*}
L\left(\frac{1}{2},\pi\otimes \chi\right)\ll_{\pi,\epsilon }p^{\frac{3}{4}-\frac{1}{32}+\epsilon }.
\end{equation*}
\end{abstract}
\vspace{5mm}
\section{Introduction}
Let $\pi$ be a $SL(3,\mathbb{Z})$ Hecke-Maass cusp form, $f$ be a $SL(2,\mathbb{Z})$ holomorphic Hecke cusp form  and $\chi$ be any non-trivial Dirichlet character $\bmod \,p$. Throughout this paper, $p$ always denotes a prime number. The Rankin-Selberg convolution of $\pi$ with $f\times\chi$ is given by
\begin{equation}\label{rank}
L(s,\pi\times f\times \chi)=\mathop{\sum\sum}_{r,n=1}^{\infty}\frac{\lambda_{\pi}(r,n)\lambda_{f}(n)\chi(nr^2)}{(nr^2)^s},
\end{equation} 
which converges absolutely for $\Re(s)>1$. The series extends to an entire function and satisfies a functional equation of the Riemann type and has arithmetic conductor $p^6$ (see \cite{booker}, Lemma 3.8 and Theorem 3.9). Consequently, the Phragmen-Lindel\"{o}f principle yields the convexity bound
\begin{equation*}
L(1/2,\pi\times f\times \chi)\ll_{\pi,f,\epsilon}p^{\frac{3}{2}+\epsilon}.
\end{equation*} The Lindel\"{o}f hypothesis asserts that the exponent $3/2+\epsilon$ can be replaced by any positive number. But even breaking the convexity barrier is hard and has remained open so far. In this paper we prove the following subconvex bound.
\begin{theorem}\label{thm}
Let $\pi$ be a $SL(3,\mathbb{Z})$ Hecke-Maass cusp form, $f$ be a $SL(2,\mathbb{Z})$ holomorphic Hecke cusp form or Hecke-Maass cusp form and $\chi$ be any non-trivial character $\bmod \,p$, then 
\begin{equation}
L\left(\frac{1}{2},\pi\times f\times\chi\right)\ll_{\pi,f,\epsilon} p^{\frac{3}{2}-\frac{1}{16}+\epsilon}.
\end{equation}
\end{theorem} We note that  Lin-Michel-Sawin \cite{lin} generalised Theorem \ref{thm} by replacing $\chi$ by a generic \textit{trace function}.
\\
Note that one has
\begin{equation}\label{twist1}
L^2(s,\pi\otimes \chi)=\left(\sum_{n=1}^{\infty}\frac{\lambda_{\pi}(1,n)\chi(n)}{n^s}\right)^2=\mathop{\sum\sum}_{r,n=1}^{\infty}\frac{\lambda_{\pi}(r,n)d(n)\chi(nr^2)}{(nr^2)^s},
\end{equation} and 
\begin{equation}\label{twist2}
L^3(s,f\otimes \chi)=\left(\sum_{n=1}^{\infty}\frac{\lambda_f(n)\chi(n)}{n^s}\right)^3=\mathop{\sum\sum}_{r,n=1}^{\infty}\frac{\lambda_{\min}(r,n)\lambda_f(n)\chi(nr^2)}{(nr^2)^s},
\end{equation} where 
\begin{equation*}
\lambda_{\min}(r,n)=\sum_{d|(r,n)}\mu (d)d_3(r/d)d_3(n/d)
\end{equation*} are the coefficients of the minimal Eisenstein series for $SL(3,\mathbb{Z})$. These are same as \eqref{rank} when $f=E(z,1/2)$ or when $\pi$ is the minimal Eisenstein series for $SL(3,\mathbb{Z})$. In fact, our proof goes through for these variations of $L$-functions as well. All we need is the fact that \eqref{twist1} and \eqref{twist2} satisfies a functional equation with arithmetic conductor $p^6$ and the corresponding Voronoi summation formulas for $d(n)$ and $d_3(n)$. Hence, we also obtain the following, which is an improvement over the subconvex bound obtained in Holowinsky-Nelson \cite{nelson} and in  \cite{munshi9} :
\begin{theorem}\label{•}
Let $\pi$ be a $SL(3,\mathbb{Z})$ Hecke-Maass cusp form and $\chi$ be any non-trivial character $\bmod\,p$, then
\begin{equation*}
L(1/2,\pi\otimes \chi)\ll_{\pi,\epsilon }p^{3/4-1/32+\epsilon }.
\end{equation*}
\end{theorem}

Let us briefly recall the history of subconvexity of $L$-functions. The first subconvex bound ($t$-aspect) for the Riemann Zeta function was obtained by G.H. Hardy and J.E. Littlewood, based on the work of Weyl  \cite{weyl}. The $q$-aspect subconvexity was first proved by D.A. Burgess. Using cancellation in character sums in short intervals, he proved that $$L\left(\frac{1}{2},\chi\right)\ll_{\epsilon} q^{\frac{3}{16}+\epsilon}. $$
D.R. Heath-Brown \cite{heath} proved the hybrid subconvexity, in both $(q,t)$-aspect, for Dirichlet $L$-functions. Since then, several improvements have been made of the above subconvex bounds.

For $GL(2)$ $L$-functions, $t$-aspect subconvexity was first proved by A. Good  \cite{good} for holomorphic forms using the spectral theory of automorphic forms. T. Meurman  \cite{meurman} then proved the result for Maass cusp forms. The $q$-aspect subconvexity for $GL(2)$ $L$-functions was first obtained by Duke-Friedlander-Iwaniec using a new form of circle method. Assuming $\chi$ to be primitive modulo $q$ and $\Re(s)=1/2$, they obtained\[L(f\times\chi,s)\ll_{f} |s|^2q^{5/11}d^2(q)\log q.\] V. Blomer and G. Harcos \cite{harcos} obtained the Burgess exponent $3/8$ for a more general holomorphic or a Maass cusp form.

For degree three $L$-functions, it was initially solved for several special cases in  \cite{blomer},\cite{munshi4}, \cite{munshi5}, \cite{munshi6} (see \cite{li} for $t$-aspect). However those only dealt with forms which are lifts of $GL(2)$ forms. In his series of papers  \cite{munshi2},\cite{munshi7},\cite{munshi8},\cite{munshi9},\cite{munshi10}, Munshi introduced a different approach to subconvexity through which he obtained subconvexity for more general degree three $L$-functions. Subconvexity bounds for Rankin--Selberg $L$-functions on $GL(2)\times GL(2)$ were known due to Sarnak \cite{sarnak}, Kowalski–- Michel-–Vanderkam \cite{kowalski}, and Lau–-Liu–-Ye \cite{liu}. The $t$-aspect subconvexity for genuine $GL(4)$ $L$-functions remains an important open problem.

Subconvexity bound for Rankin--Selberg $GL(3)\times GL(2)$ $L$-functions was first obtained by Li \cite{li}, where she considered the $L$-function of a fixed self-dual Hecke-Maass form for $SL(3,\mathbb{Z})$ twisted by a Hecke-Maass form for  $SL(2,\mathbb{Z})$, or by a holomorphic Hecke cusp form for $SL(2,\mathbb{Z})$, in the eigenvalue aspect, or respectively in the weight aspect, of the $GL(2)$ form. Blomer \cite{blomer} considered this problem in the level aspect where the twist is by special Hecke-Maass forms of prime square level.

Recently, using his separation of oscillation technique (as introduced in \cite{munshi7}), $t$-aspect subconvexity  for $GL(3)\times GL(2)$ $L$- functions was obtained by Munshi \cite{munshi1}, where the $GL(3)$ form is any Hecke-Maass cusp form for $SL(3, \mathbb{Z})$.  Following similar strategy,  we use the `conductor lowering' trick as a device for separation of oscillations in the circle method as introduced in \cite{munshi2}. The key feature here is that the $GL(2), GL(3)$ Voronoi summations together transform the Ramanujan sum arising in the delta method to an additive character with respect to the $GL(3)$ variable. This was the crucial observation in \cite{munshi1}. As such, we save more after applying Poisson summation after Cauchy-Schwarz inequality. After a final application of the Cauchy-Schwarz inequality followed by Poisson summation, we observe that there is not enough saving in the `zero frequency' and extra saving in the non-zero frequencies. To optimize this, we use the `Mass transfer' trick introduced in \cite{munshi3} which essentially decreases the contribution of the zero frequency at the cost of an increase in the length of non-zero frequencies. This is worth comparing to the trick used by James Maynard in \cite{maynard} (Section 3.4), where he is faced with the same problem of reducing the size of diagonal terms. He handles this by artificially introducing congruence conditions at the beginning, whereas we achieve it by increasing the size of the equation that we detect using the delta symbol (see \eqref{eq3} below). We proceed to prove the theorem for $f$ a Maass cusp form. Note that the same proof goes through for other cases as well with mild alterations (depending on their respective Voronoi summation formulas to be precise).\\
\section{The set up }
Let $\pi$ be a Hecke-Maass cusp form of type $(\nu_1,\nu_2)$ for $SL(3,\mathbb{Z})$, $f$ be a Maass cusp form of type $1/2+i\nu$ for $SL(2,\mathbb{Z})$ and $\chi$ be any non-trivial character $\bmod \,p$. The associated $L$-function $L(s,\pi\times f\times\chi)$ has arithmetic conductor $p^6$. Hence, from the approximate functional equation (see Th. 5.3 and Prop. 5.4 in \cite{iwaniec}, Theorem 3.9 in \cite{booker}) we get
\begin{equation}\label{FE}
L\left(\frac{1}{2},\pi\times f\times\chi\right)\ll\left|\sum_{m=1}^{\infty}\frac{\lambda_{\pi\times f\times\chi}(m)}{m^{1/2}}\mathcal{V}\left({\frac{m}{p^{3}}}\right)\right|,
\end{equation}
where $$\lambda_{\pi\times f\times\chi}(m)=\mathop{\sum\sum}_{nr^2=m}\lambda_{\pi}(r,n)\lambda_{f}(n)\chi(nr^2)$$
and the smooth function $\mathcal{V}$ satisfies
\begin{equation}\label{22}
x^j\mathcal{V}^{(j)}(x)=O_{A}(1+|x|)^{-A}.
\end{equation}
\begin{remark}[Notation]
In this paper the notation $\alpha\ll A$ will mean that for any $\epsilon>0$, there is a constant $c$ such that $|\alpha|\leq cAp^{\epsilon}$. The dependence of the constant on $\pi,f$ and $\epsilon$, when occurring, will be ignored. We will also use the phrase ``negligible error" by which we mean an error term $O_{B}(p^{-B})$ for an arbitrary $B>0$.
\end{remark}
We reduce our problem to estimating the following partial sums :
\begin{lemma}\label{•}For any $0\leq\theta \leq 3$, we have
\begin{equation}\label{dyd}
L\left(\frac{1}{2},\pi\times f\times\chi\right)\ll \sup_{r\leq p^{\theta}}\sup_{\frac{p^{3-\theta}}{r^2}\leq N\leq \frac{p^{3+\epsilon}}{r^2}}\frac{|S_r(N)|}{N^{1/2}}+p^{(3-\theta)/2},
\end{equation}
where
 \begin{equation*}
S_r(N)=\sum_{n=1}^{\infty}\lambda_{\pi}(r,n)\lambda(n)\chi(n)V\left(\frac{n}{N}\right)
\end{equation*} where $V$ is a smooth function supported in $[1,2]$ and satisfies $V^{(j)}(x)\ll_{j} 1$.

\end{lemma}	
\begin{proof}From \eqref{FE} and \eqref{22} we get
\begin{equation}\label{rr}
\begin{aligned}
L\left(\frac{1}{2},\pi\times f\times\chi\right)&\ll\left|\mathop{\sum\sum}_{n,r=1}^{\infty}\frac{\lambda_{\pi}(r,n)\lambda_{f}(n)\chi(nr^2)}{(nr^2)^{1/2}}\mathcal{V}\left(\frac{nr^2}{p^{3}}\right)\right|\\
&\ll\left|\mathop{\sum\sum}_{nr^2\leq p^{3+\epsilon}}\frac{\lambda_{\pi}(r,n)\lambda_{f}(n)\chi(nr^2)}{(nr^2)^{1/2}}\mathcal{V}\left(\frac{nr^2}{p^{3}}\right)\right|+p^{-2019}\\
&=\left|\sum_{r\leq p^{(3+\epsilon)/2}}\frac{\chi(r^2)}{r}\sum_{n\leq p^{3+\epsilon}/r^2}\frac{\lambda_{\pi}(r,n)\lambda_{f}(n)\chi(n)}{n^{1/2}}\mathcal{V}\left(\frac{nr^2}{p^{3}}\right)\right|+p^{-2019}.
\end{aligned} 
\end{equation}
We divide the range in the last summation as
\begin{equation*}
\sum_{r\leq p^{(3+\epsilon)/2}}\sum_{n\leq\frac{ p^{3+\epsilon}}{r^2}}=\sum_{r\leq p^{\theta}}\sum_{\frac{p^{3-\theta}}{r^2}\leq n\leq \frac{p^{3+\epsilon}}{r^2}}+\sum_{r\leq p^{\theta}}\sum_{n< \frac{p^{3-\theta}}{r^2}}+\sum_{r>p^{\theta}}\sum_{n\leq\frac{ p^{3+\epsilon}}{r^2}},
\end{equation*}where an optimal $\theta>0$ will be chosen later. Using the Ramanujan bound on average
\begin{equation}\label{rama}
\mathop{\sum\sum}_{n_1^2n_2\leq x}|\lambda_{\pi}(n_1,n_2)|^2\ll x,\,\,\,\,\quad\sum_{n\leq x}|\lambda_f(n)|^2\ll x 
\end{equation} one sees that the last two ranges contributes at most $p^{(3-\theta)/2}$ to \eqref{rr}. Hence we have
\begin{equation*}
L\left(\frac{1}{2},\pi\times f\times\chi\right)\ll\left|\sum_{r\leq p^{\theta}}\frac{\chi(r^2)}{r}\sum_{\frac{p^{3-\theta}}{r^2}\leq n\leq \frac{p^{3+\epsilon}}{r^2}}\frac{\lambda_{\pi}(r,n)\lambda_{f}(n)\chi(n)}{n^{1/2}}\mathcal{V}\left(\frac{nr^2}{p^{3}}\right)\right|+p^{(3-\theta)/2}.
\end{equation*} Using a smooth dyadic partition of unity $W$, we see that the second sum inside the absolute value above is at most
\begin{equation*}
 \sup_{\frac{p^{3-\theta}}{r^2}\leq N\leq \frac{p^{3+\epsilon}}{r^2}}\Bigg|\sum_{n=1}^{\infty}\frac{\lambda_{\pi}(r,n)\lambda_{f}(n)\chi(n)}{n^{1/2}}W\left(\frac{n}{N}\right)\mathcal{V}\left(\frac{nr^2}{p^{3}}\right)\Bigg|,
\end{equation*} which can we written as
 \begin{equation}\label{par}
\sup_{\frac{p^{3-\theta}}{r^2}\leq N\leq \frac{p^{3+\epsilon}}{r^2}}\frac{|S_r(N)|}{N^{1/2}},
\end{equation}where 
\begin{equation*}
S_r(N)=\sum_{n=1}^{\infty}\lambda_{\pi}(r,n)\lambda(n)\chi(n)V_{r,N}\left(\frac{n}{N}\right).
\end{equation*} where $V_{r,N}(x)=x^{-1/2}W(x)\mathcal{V}(Nr^2x/p^{3})$. Note that $V_{r,N}(x)$ is supported on $[1,2]$ and satisfies $V_{r,N}^{(j)}\ll_{j} 1$ (bounds independent of $r,N$). Henceforth we ignore the dependence on $r,N$ and assume $V_{r,N}$ is same function for all $r,N$ and call it $V(x)$ (abusing notation). The claim follows.
\end{proof}

\subsection{The delta method.} We now separate oscillations from $\lambda_{\pi}(r,n)$ and $\lambda (n)\chi(n)$ using a version of the delta method due to Duke, Friedlander and Iwaniec. More specifically, we will use the expansion $(20.157)$ given in Chapter 20 of \cite{iwaniec}. Let $\delta : \mathbb{Z}\to \{0,1\}$ be defined by
\[
\delta(n)=
\begin{cases}
1 &\text{if}\,\,n=0 \\
0 &\text{otherwise}.
\end{cases}
\]
Then for $n\in\mathbb{Z}\cap [-2M,2M]$, we have
\begin{equation}\label{deltasymbol}
\delta(n)=\frac{1}{Q}\sum_{q\leq Q}\frac{1}{q}\sum_{a\bmod q}e\left(\frac{na}{q}\right)\int_{\mathbb{R}}g(q,x)e\left(\frac{nx}{q	Q}\right) dx \,\,,
\end{equation} where $Q=2M^{1/2}$. The function $g$ satisfies the following property (see $(20.158)$ and $(20.159)$ of \cite{iwaniec}, Lemma 15 of \cite{huang}).
\begin{equation}\label{delta}
\begin{aligned}
&g(q,x)=1+O\left(\frac{Q}{q}\left(\frac{q}{Q}+|x|\right)^A\right),\,\,\,\, \,\,\,\, g(q,x)\ll |x|^{-A}\\
&x^j\frac{\partial^j}{\partial x^j} g(q,x)\ll \log Q\min \{Q/q,1/|x|\}
\end{aligned}
\end{equation}for any $A>1$. In particular, the second property implies that the effective range of the integral in \eqref{deltasymbol} is $[-M^{\epsilon},M^{\epsilon}]$. It also follows from first the property that if $\min \{Q/q,1/|x|\} > Q^{\epsilon}$, we can replace $g$ by $1$ up to a negligible error term. And if $\min \{Q/q,1/|x|\} \leq Q^{\epsilon}$, then the third property implies $x^j\frac{\partial^j}{\partial x^j} g(q,x)\ll Q^{\epsilon}$. In any case, we can view $g(q,x)$ as a `nice' weight function.
\vspace{5mm}
\subsection{The mass transfer} A direct application of the delta method fails to beat the convexity bound for the diagonal contribution. So, we need employ the trick of `Mass transfer' (amplification) as introduced by Munshi in \cite{munshi3}.

 From above discussion, it suffices study to the sums $S_r(N)$. Let $\mathscr{L}$ be the set of primes in $[L,2L]$. Recall the Hecke relation (see Theorem 6.4.11 of \cite{goldfeld})
\begin{equation*}
\lambda_{\pi}(1,l)\lambda_{\pi}(r,n)=\lambda_{\pi}(r,nl)+\lambda_{\pi}(rl,n/l)+\lambda_{\pi}(r/l,nl)\,,
\end{equation*}where the second term occurs only if $l|n$ and the third term only if $l|r$. The latter does not happen since the size of  $l\in\mathcal{L}$ and $r$ will be chosen such that
\begin{equation}\label{eq2.8}
r<L=p^{\eta}\,
\end{equation}
for some $0<\eta<1$. Using this we have
\begin{equation*}
S_r(N)=\frac{1}{\sum_{l\in\mathscr{L}}|\lambda_{\pi}(1,l)|^2}\sum_{l\in\mathscr{L}}\overline{\lambda_{\pi}(1,l)}\sum_{n=1}^{\infty}(\lambda_{\pi}(r,nl)+\lambda_{\pi}(rl,n/l))\lambda(n)\chi(n)V\left(\frac{n}{N}\right).
\end{equation*}
Using the fact that (see Lemma 5.1 of \cite{liu1})
\begin{equation*}
\sum_{l\in\mathscr{L}}|\lambda_{\pi}(1,l)|^2\gg L^{1-\epsilon},
\end{equation*}  we get
\begin{equation*}
S_r(N)\ll\frac{1}{L}\sum_{l\in\mathscr{L}}\left|\overline{\lambda_{\pi}(1,l)}\sum_{n=1}^{\infty}(\lambda_{\pi}(r,nl)+\lambda_{\pi}(rl,n/l))\lambda(n)\chi(n)V\left(\frac{n}{N}\right)\right|.
\end{equation*}
The contribution coming from the second term is
\begin{equation}\label{2.9}
\frac{1}{L}\sum_{l\in \mathcal{L}}|\overline{\lambda_{\pi}(1,l)}|\sum_{n'\sim N/l}|\lambda_{\pi}(rl,n')|\left(|\lambda(n')\lambda(l)|+|\lambda(n'/l)|\right).
\end{equation}Note that from  Hecke relations one has
\begin{equation}\label{2.10}
\begin{aligned}
\sum_{b\ll X}|\lambda_{\pi}(a,b)|^2=&\sum_{b\ll X}\left|\sum_{d|(a,b)}\mu(d)\lambda_{\pi}(a/d,1)\lambda_{\pi}(1,b/d)\right|^2\\
\ll &\sum_{ d|a}|\lambda_{\pi}(a/d,1)|^2\sum_{b'\ll X/d}|\lambda_{\pi}(1,b')|^2\\
\ll& a^{2\theta_0+\epsilon}X,
\end{aligned}
\end{equation}where $\theta_0=5/14$ (see \cite{kimsar}). By Cauchy-Schwarz, \eqref{2.10} and Ramanujan bound on average for the $GL(2)$ coefficient, we see that \eqref{2.9} is bounded by
\begin{equation}\label{err1}
\frac{(rL)^{\theta_0}N}{L^2}\sum_{l\in\mathcal{L}}|\lambda_{\pi}(1,l)|(|\lambda (l)|+1)\ll \frac{(rL)^{\theta_0}N}{L}\ll \frac{r^{1/2}N}{L^{1/2}}.
\end{equation}Hence we get
\begin{equation}\label{eq212}
S_r(N)\ll\frac{1}{L}\left|\sum_{l\in\mathscr{L}}\overline{\lambda_{\pi}(1,l)}\sum_{n=1}^{\infty}\lambda_{\pi}(r,nl)\lambda(n)\chi(n)V\left(\frac{n}{N}\right)\right|+O\left(\frac{r^{1/2}N}{L^{1/2}}\right).
\end{equation}Write
\begin{equation}
\tilde{S}_r(N)=\frac{1}{L}\sum_{l\in\mathscr{L}}\overline{\lambda_{\pi}(1,l)}\sum_{n=1}^{\infty}\lambda_{\pi}(r,nl)\lambda(n)\chi(n)V\left(\frac{n}{N}\right)
\end{equation}The rest of the paper is devoted to estimation of the sum $\tilde{S}_r(N)$. Plugging in the $\delta$ function, we get
\begin{equation}\label{eq3}
\begin{aligned}
\tilde{S}_r(N)=\frac{1}{L}\sum_{l\in\mathscr{L}}\overline{\lambda_{\pi}(1,l)}\sum_{m}\sum_{\substack{n\\p|(m-nl)}}\lambda_{\pi}(r,m)V\left(\frac{m}{lN}\right)\delta\left(\frac{m-nl}{p}\right)&\lambda(n)\chi(n)U\left(\frac{n}{N}\right)
\end{aligned}
\end{equation}where $U$ is any smooth function supported in $[1/2,5/2]$ and equals 1 on $[1,2]$. To complete the separation, we use the delta expansion in \eqref{deltasymbol} with 
\begin{equation}
Q=(NL/p)^{1/2}
\end{equation}
and detect $p|(m-nl)$ using additive characters to see that  
\begin{equation}\label{eq5}
\begin{aligned}
&\tilde{S}_r(N)\\
&=\frac{1}{pQL}\int_{\mathbb{R}}\sum_{l\in\mathcal{L}}\overline{\lambda_{\pi}(1,l)}\sum_{u=0}^{p-1}\sum_{1\leq q\leq Q}\frac{g(q,x)}{q}\sideset{}{^*}\sum_{a (q)}\left(\sum_m\lambda_{\pi}(r,m)e\left(\frac{ma}{pq}+\frac{mx}{pqQ}+\frac{mu}{p}\right)V\left(\frac{m}{lN}\right)\right)\\
&\quad\quad\quad\quad\quad\quad\quad\quad\quad\quad\quad\quad\quad\quad\quad\times \left(\sum_{n}\lambda(n)\chi(n)e\left(-\frac{nla}{pq}-\frac{nlx}{pqQ}-\frac{lnu}{p}\right)U\left(\frac{n}{N}\right)\right)dx\,.
\end{aligned}
\end{equation}
We first work out the details assuming
\begin{equation}\label{210}
(pl,q)=1  \,\,\,\text{and}\,\,\,\,u\neq 0.
\end{equation}The remaining cases provide us with smaller bounds and are dealt with later towards the end of the paper. In the case \eqref{210}, we replace $a\to pa $ and re-write \eqref{eq5}  as 
\begin{equation}\label{217}
\begin{aligned}
&\tilde{S}_r(N)\\
&=\frac{1}{pQL}\int_{\mathbb{R}}\sum_{l\in\mathcal{L}}\overline{\lambda_{\pi}(1,l)}\sum_{u=0}^{p-1}\sum_{\substack{1\leq q\leq Q\\(pl,q)=1}}\frac{g(q,x)}{q}\sideset{}{^*}\sum_{a (q)}\left(\sum_m\lambda_{\pi}(r,m)e\left(\frac{ma}{q}+\frac{mx}{pqQ}+\frac{mu}{p}\right)V\left(\frac{m}{lN}\right)\right)\\
&\quad\quad\quad\quad\quad\quad\quad\quad\quad\quad\quad\quad\quad\quad\quad\times \left(\sum_{n}\lambda(n)\chi(n)e\left(-\frac{nla}{q}-\frac{nlx}{pqQ}-\frac{lnu}{p}\right)U\left(\frac{n}{N}\right)\right)dx\,.
\end{aligned}
\end{equation}
\vspace{5mm}
\section{Sketch of the proof}
For simplicity we assume the generic case $N\sim p^3, q\sim Q=p\sqrt{L}$ and $r=1$. After applying the circle method and the conductor lowering trick by Munshi , our main object of study becomes of the form
\begin{equation}
\begin{aligned}
\frac{1}{pqQL}\sum_{u\bmod p}\sum_{q\sim Q}\sum_{a\bmod q}\sum_{l\sim L}\overline{\lambda_{\pi}(1,l)}&\sum_{m\sim NL}\lambda_{\pi}(1,m)e\left(\frac{m(ap+uq)}{pq}\right)\\
&\sum_{n\sim N}\lambda(n)\chi(n)e\left(\frac{-nl(ap+uq)}{pq}\right).
\end{aligned}
\end{equation} Trivially estimating at this stage gives $S(N)\ll N^2L$. So we want to save $NL$ plus a little more in the above sum.  We apply Voronoi summation formulae to both $n$ and $m$ sums. In the $GL(2)$ Voronoi formula we save $N/pq\sim p/\sqrt{L}$ with the dual length $\sim pL$ and in the $GL(3)$ Voronoi formula we save $NL/(pq)^{3/2}\sim L^{1/4}$ with the dual length $\sim p^3\sqrt{L}$. We also save  $\sqrt{Q}$ in the $a$ sum and $\sqrt{p}$ in the $u$ sum. Hence in total we have saved $p^2$ and at this stage trivial estimation gives
\begin{equation}
S(N)\ll \frac{(pQ)^{5/2}}{\sqrt{Q}\sqrt{p}}=p^4L.
\end{equation} So we need to save $pL$ and a little more in the tranformed sum 
\begin{equation}
\sum_{q\sim Q}\sum_{m\sim p^3\sqrt{L}}\lambda_{\pi}(m,1)\sum_{n\sim pL}\lambda(n)\mathcal{C} \,\mathcal{J},
\end{equation} where $\mathcal{J}$ is a integral transform which does not oscillate and 
\begin{equation}\label{add}
\begin{aligned}
\mathcal{C}&\approx\sum_{b (p)}\bar{\chi}(b)\sum_{u (p)}\sum_{a(q)}S(\overline{ap+uq},m;pq)e\left(\frac{n \overline{l(ap+(u-b\bar{l})q)}}{pq}\right)\\
&\leadsto q\sum_{\alpha (p)} \mathcal{C}_1(q,\alpha,-n\bar{l})e\left(\frac{(\overline{-n\bar{l}+q\alpha})m}{pq}\right),
\end{aligned}
\end{equation}where 
\begin{equation}
\mathcal{C}_1(q,\alpha,n)=\sum_{b (p)}\bar{\chi} (b)\sum_{u (p)}e\left(\frac{\bar{q}^2((n+q\alpha)\bar{u}+n(\overline{u-b\bar{l}}))}{p}\right).
\end{equation}
Note that we have already assumed square root cancellation in all the variables in $\mathcal{C}$.

We next apply Cauchy-Schwarz inequality to arrive at
\begin{equation}\label{36}
p^{3/2}L^{1/4}\left(\sum_{m\sim p^3\sqrt{L}}\left|\sum_{\alpha (p)}\sum_{l\sim L}\sum_{q\sim Q}\sum_{n\sim pL}\lambda(n)\mathcal{C}_1(q,\alpha,-n\bar{l}) e\left(\frac{(\overline{-n\bar{l}+q\alpha})m}{pq}\right)\mathcal{J}\right|^2\right)^{1/2}.
\end{equation}Opening the absolute value square, we apply Poisson summation formula on the $m$ sum. In the diagonal (zero frequency) we save $(LQpL)^{1/2}\sim pL^{5/4}$ and the contribution of the diagonal becomes $p^4L\times 1/pl^{5/4}=p^3/L^{1/4}$. Recall that the square-root saving in $\alpha\bmod p$ have been already considered. By linearity, the off-diagonal produces a factor $pqq'$ together with a congruence relation of the form
\begin{equation}
q'(\overline{-n\bar{l}+q\alpha})-q(\overline{-n'\bar{l'}+q'\alpha'})= n_2 (\bmod pqq') .
\end{equation}
We save $q$ and $q'$ from $n$ and $n'$ respectively. Given $l,l'q,q',n$ and $n'$, $\alpha'$ is then determined by $\alpha$, say $f(\alpha)$. Finally, by viewing as a hypergeometric function (as defined by Katz), square-root cancellation in the sum
\begin{equation}
\sum_{\alpha (p)}\mathcal{C}_1(q,\alpha,-n\bar{l})\mathcal{C}_1(q',f(\alpha),-n'\bar{l'}).
\end{equation}is proved by Will Sawin in the appendix. Before Cauchy-Schwarz, the $\alpha$ and $m$ sum were contributing $p^3\sqrt{L}\times \sqrt{p}$. After Cauchy-Schwarz and Poisson summation, they are contributing $p^{3/2}L^{1/4}p^{1/4}p^{1/2}$. Hence in the off-diagonal, we effectively save $p^{3/2}L^{1/4}/p^{1/4}$ and its contribution to $S(N)$ is bounded by $ p^4L\times 1/p^{5/4}L^{1/4}=p^{(3-1/4)}L^{3/4}$. With the optimal choice $L=p^{1/4}$, we get the claimed bound. 
\section{Voronoi summation formulae}
We now proceed to transform the $GL(3)$ and $GL(2)$ sums in \eqref{eq5} using their respective Voronoi formulas.
\subsection{The $GL(3)$ sum.}Let $\{\alpha_i : i=1,2,3\}$ be the Langlands parameters for $\pi$. Let $g$ be a compactly supported smooth function on $(0,\infty)$. For $l=0,1$, we define
\begin{equation*}
\gamma_l(s)=\frac{{\pi}^{-3s-\frac{3}{2}}}{2}\prod_{i=1}^{3}\frac{\Gamma\left(\frac{1+s+\alpha_i+l}{2}\right)}{\Gamma\left(\frac{-s-\alpha_i+l}{2}\right)}.
\end{equation*}Set $\gamma_{\pm}(s)=\gamma_{0}\pm i\gamma_{1}(s)$ and let
\begin{equation}\label{41}
G_{\pm}(y)=\frac{1}{2\pi i}\int_{(\sigma)}y^{-s}\gamma_{\pm}(s)\tilde{g}(-s)ds=G_0(y)\pm iG_{1}(y)\,,
\end{equation}where $\sigma> -1+\max\{-\Re(\alpha_1),-\Re(\alpha_2),-\Re(\alpha_3)\}$ and $\tilde{g}$ is the Mellin transform of $g$.
\begin{lemma}[\textbf{$GL(3)$ Voronoi summation formula}]Let $c,d,r\geq 1$ be integers such that $(c,d)=1$. Then, with the above notation we have,
\begin{equation}\label{vor}
\begin{aligned}
&\sum_{n=1}^{\infty}\lambda_{\pi}(r,n)e\left(\frac{dn}{c}\right)g(n) \\
&=c\sum_{\pm}\sum_{n_1|cr}\sum_{n_2=1}^{\infty}\frac{\lambda_{\pi }(n_1,n_2)}{n_1n_2}S(r\bar{d},\pm n_2; cr/n_1)G_{\pm}\left(\frac{n_1^2n_2}{c^3r}\right).
\end{aligned}
\end{equation} 
\end{lemma}
\begin{proof} See \cite{voronoi}
\end{proof}
We have the following oscillatory behaviour of transforms $G_{\pm}$  due to X. Li \cite{li}.
\begin{lemma}\label{•}
Suppose $g(y)$ is a smooth function  compactly supported on $[Y,2Y]$. Then for any fixed integer $K\geq 1$ and $yY\gg1$,
\begin{equation*}
G_0(y)=\frac{{\pi}^{\frac{3}{2}}y}{2}\int_{0}^{\infty}g(z)\sum_{j=1}^{K}\frac{c_j\cos (6\pi y^{\frac{1}{3}}z^{\frac{1}{3}})+d_j\sin (6\pi y^{\frac{1}{3}}z^{\frac{1}{3}})}{{({\pi}^3yz)}^{\frac{j}{3}}}dz+O\left((yY)^{\frac{-K+2}{3}}\right)\,,
\end{equation*}where $c_j$ and $d_j$ are constants depending on $\alpha_i$'s.
\end{lemma}The asymptotic behaviour $G_1(y)$ is similar, with changes only in the constants $c_j$ and $d_j$. In $G_0(y)$, note that the oscillatory part is same for all $j$, whereas the magnitude of weight functions $g(z)/(yz)^{j/3}$ decreases as $j$ increases (this uses the assumption $yY\gg 1$). Hence the main term in $G_0(y)$ is well represented by the $j=1$ contribution, and we will be working with this term only.   
\\
\\
\noindent
Note that $(pq, ap+uq)=1$ by the assumption in \eqref{210}. In our context we have $c=pq, d=ap+uq $ and $g(n)=e(nx/pqQ)V(n/lN)$.  By the assumption in Lemma \ref{vor}, the above asymptotic expansion can be used in the range
\begin{equation}
\frac{NLn_1^2n_2}{(pq)^3r}\gg p^{\epsilon}
\end{equation}i.e,
\begin{equation}\label{range}
n_1^2n_2\gg \frac{p^{\epsilon}(pq)^3r}{NL},
\end{equation}
where are able to replace $m$ sum in \eqref{eq5} essentially by 
\begin{equation}\label{msum}
\begin{aligned}
\frac{(Nl)^{2/3}}{pqr^{2/3}}\sum_{\pm}\sum_{n_1|pqr}n_1^{1/3}\sum_{n_2=1}^{\infty}\frac{\lambda_{\pi}(n_1,n_2)}{n_2^{1/3}}&S\left(r(\overline{ap+uq}),\pm n_2;pqr/n_1\right)\\ &\times
\int_{\mathbb{R}}V(z)e\left(\frac{Nlxz}{pqQ}\pm \frac{3(Nln_1^2n_2z)^{1/3}}{pqr^{1/3}}\right)dz \,.
\end{aligned}
\end{equation}By repeated integration by parts, we see that the integral is negligibly small unless $n_1^2n_2\ll M_0$, where
\begin{equation}
M_0=p^{\epsilon}rN^2L^2/Q^3=rp^{3/2+\epsilon}N^{1/2}L^{1/2}.
\end{equation}
Note that if $q\leq p^{-\epsilon}Q$, then
\begin{equation}
\frac{Nl}{pqQ}\gg p^{\epsilon},
\end{equation}in which case the integral above is seen to be negligibly small unless \begin{equation}
n_1^2n_2\gg p^{-2\epsilon}M_0
\end{equation}On other hand if $q\geq p^{-\epsilon}Q$, then \eqref{range} gives
\begin{equation}
n_1^2n_2\gg \frac{p^{-2\epsilon}(pQ)^3r}{NL}=p^{-3\epsilon}M_0
\end{equation}
Hence we use expression \eqref{msum} in the range
\begin{equation}
 p^{-3\epsilon}M_0\ll n_1^2n_2\ll M_0= rp^{3/2+\epsilon}N^{1/2}L^{1/2}=\frac{p^{\epsilon}(pQ)^3r}{NL}.
\end{equation}
\begin{remark}
In the complementary range $xX\ll 1$, i.e.
\begin{equation}\label{comp}
n_1^2n_2\ll \frac{p^{\epsilon}(pq)^3r}{NL} :=\tilde{M}_0,
\end{equation}
 by shifting the contour in \eqref{41} to $\sigma=-1/3$ (say) without crossing the poles of the gamma factors we get
\begin{equation}
x^jG_{\pm}^{(j)}(x)\ll_j (xX)^{1/3}\ll_j 1.
\end{equation}So, in this case, we put the integral transform $G_{\pm}\left(\frac{n_1^2n_2}{c^3r}\right)$ inside the weight functions and proceed same as below. The only difference is now we will be working with a smaller range for the $n_2$ sum given by \eqref{comp}, and a smaller range for the Poisson variable at a later stage (see Remark \ref{Remark 4} below).
\end{remark}
\vspace{5mm}
\subsection{The $GL(2)$ sum} 
\begin{lemma}[\textbf{GL(2) Voronoi summation formula.}]Let $f$ be a Maass cusp form with Laplacian eigenvalue $1/4+\nu^2, \nu \geq 0$, and $\varepsilon_f$ the eigenvalue of the involution operator. Let $h$ be a compactly supported smooth function on the interval $(0,\infty)$. Let $a>0$ an integer and $c\in\mathbb{Z}$ be such that $(a,c)=1$. Then we have

\begin{equation}\label{vor2}
\sum_{n=1}^{\infty}\lambda_{f}(n)e\left(\frac{an}{c}\right)h(n)=\frac{1}{c}\sum_{\pm}\sum_{n=1}^{\infty}\lambda_{f}(n)e\left(\frac{\pm\overline{a}n}{c}\right)H^{\pm}\left(\frac{n}{c^2}\right)\,,
\end{equation}
where for $\nu>0$,
\begin{equation*}
\begin{aligned}
& H^{-}(\alpha)=\frac{-\pi}{\sin (\pi i\nu)}\int_{0}^{\infty}h(y)\{J_{2i\nu}-J_{-2i\nu }\}(4\pi\sqrt{y\alpha})dy \,,\\
& H^{+}(\alpha)=4\varepsilon_f\cosh (\pi\nu)\int_{0}^{\infty}h(y)K_{2i\nu}(4\pi\sqrt{y\alpha})dy\,,
\end{aligned}
\end{equation*}and for $\nu=0$,
\begin{equation*}
H^{-}(\alpha)=-2\pi\int_{0}^{\infty}h(y)Y_0(4\pi\sqrt{y\alpha})dy,\,\,\,\text{and}\,\,\, H^{+}(\alpha)=4\varepsilon_f\int_{0}^{\infty}h(y)K_{0}(4\pi\sqrt{y\alpha})dy\,.
\end{equation*}
\end{lemma}
\begin{proof}
See appendix A.4 of \cite{kowalski}.
\end{proof}

As with the $GL(3)$ case, we need the following asymptotics for the Bessel functions to extract the oscillation (see \cite{watson}, p. 206).
\begin{lemma}For $y>0$, the Bessel functions $J_{2i\nu}(y)$ and $K_{2i\nu}(y)$, ($\nu\in \mathbb{R}$) satisfy the following oscillatory behaviour 
\begin{equation}\label{bes} 
J_{2i\nu}(y)=e^{iy}P_{ 2i\nu}(y)+e^{-iy}Q_{ 2i\nu}(y)\,\,\,\hbox{and}\,\,\,\, \left|y^kK_{2i\nu}^{(k)}(y)\right|\ll_{k,\nu}\frac{e^{-y}}{\sqrt{y}}\,,
\end{equation}where the function $P_{ 2i\nu}(y)$ (and similarly $Q_{ 2i\nu}(y)$) satisfies 
\begin{equation}\label{48}
y^jP_{2i\nu}^{(j)}(y)\ll_{j,\nu}\frac{1}{\sqrt{y}}.
\end{equation}
\end{lemma}
Now, the $n$ sum in \eqref{eq5} equals
\begin{equation}\label{nsums}
\frac{1}{g_{\bar{\chi}}}\sum_{b (p)}\bar{\chi}(b)\sum_{n}\lambda(n)e\left(\frac{-nl(ap+(u-b\bar{l})q)}{pq}\right)e\left(\frac{-nlx}{pqQ}\right)U\left(\frac{n}{N}\right).
\end{equation}
Recall that $(pl,q)=1$ by assumption \eqref{210} and $(p,l)=1$ since $l$ is prime and $l\sim L<p$. Note that if $b=ul\bmod p$, then the conductor of $n$ sum reduces to $(NL)^2/(pQ)^2\ll NL/p$, whereas the initial length is of size $N$. Hence the dual length is $\ll L/p\ll p^{-\delta}$, for some $\delta>0$ to be chosen later. So, the dual sum becomes negligible in the case $b=ul\bmod p$. For $b\neq ul$, applying the $GL(2)$ Voronoi formula to \eqref{nsums} with $c=pq, a=-l(ap+(u-b\bar{l})q)$ and $h(n)=e(-nlx/pqQ)U(n/N)$, we arrive at
\begin{equation}
\begin{aligned}
\frac{1}{pqg_{\bar{\chi}}}\sum_{\substack{b (p)\\b\neq ul}}\bar{\chi}(b)\sum_{\pm}&\sum_{n=1}^{\infty}\lambda(n) e\left(\mp\frac{n\overline{l(ap+(u-b\bar{l})q)}}{pq}\right)H^{\pm}\left(\frac{n}{p^2q^2}\right).
\end{aligned}
\end{equation}We proceed with the calculation for the first part involving $H^{-}$ (see Remark \ref{re} below). Using \eqref{bes} we see that the  part with $H^{-}$ is essentially a sum of four sums of the type
\begin{equation}\label{eq8}
\begin{aligned}
\frac{N^{3/4}}{(pq)^{1/2}g_{\bar{\chi}}}\sum_{\substack{b (p)\\b\neq u}}\bar{\chi}(b)&\sum_{n}\frac{\lambda(n)}{n^{1/4}} e\left(\frac{n\overline{l(ap+(u-b\bar{l})q)}}{pq}\right)\\ 
&\times \int_{\mathbb{R}}\left(\frac{\sqrt{nN}}{pq}\right)^{1/2} U_{\pm2i\nu}\left(\frac{4\pi\sqrt{nNy}}{pq}\right)U(y)e\left(-\frac{lNxy}{pqQ}\pm\frac{2\sqrt{nNy}}{pq}\right)dy.
\end{aligned}
\end{equation}
By \eqref{48}, the weight function $(\sqrt{nN}/(pq))^{1/2}U_{\pm2i\nu}\left(\frac{4\pi\sqrt{nNy}}{pq}\right)U(y)$ has bounded derivatives (bounds not depending on $n,N,p,q$). Like earlier, we assume that this function is the same for all $n,N,p,q$ and call it $U(y)$.\\
\\
\noindent
 By repeated integration by parts, it is clear that the integral is negligibly small unless $n\ll N_0$, where
\begin{equation}
 N_0=p^{\epsilon}\frac{NL^2}{Q^2}=p^{1+\epsilon}L.
\end{equation}
\begin{remark}\label{re}
For the estimation of the part involving $H^+$, note that by the second property of \eqref{bes}, the function
\begin{equation}\label{4.20}
(\sqrt{nN}/(pq))^{1/2}K_{ 2i\nu}\left(\frac{4\pi\sqrt{nNy}}{pq}\right)
\end{equation}
has bounded derivatives. Hence, in this case we can put \eqref{4.20} inside the weight functions and proceed same as below. The only difference is the potentially smaller range for the dual sum given by
\begin{equation}
\frac{\sqrt{nN}}{pq}\leq p^{\epsilon}\,\,\implies n\ll p^{\epsilon}\frac{p^2q^2}{N}\leq p^{1+\epsilon}L=N_0,
\end{equation}which again follows from the second property of \eqref{bes}.
\end{remark}
\vspace{1mm}
\section{Cauchy and Poisson}
\subsection{Cauchy inequality}
Rearranging \eqref{eq8} we get,
\begin{equation}\label{nsum}
\frac{N^{3/4}}{(pq)^{1/2}g_{\bar{\chi}}}\sum_{n\ll N_0}\frac{\lambda(n)}{n^{1/4}}C_1(n\bar{l},a,q,u)J(n,q,l)\,,
\end{equation}where
\begin{equation}
C_1(n,a,q,u)=\sum_{\substack{b (p)\\b\neq ul}}\bar{\chi}(b)e\left(\frac{n\overline{(ap+(u-b\bar{l})q)}}{pq}\right)\,
\end{equation}and $J(n,q,l)$ is the integral in \eqref{eq8}.\\
Combining \eqref{eq5}, \eqref{msum} and \eqref{nsum}, we arrive at
\begin{equation}\label{comb}
\begin{aligned}
\tilde{S}_r(N)=\frac{N^{\frac{3}{4}+\frac{2}{3}}}{g_{\bar{\chi}}p^{\frac{5}{2}}r^{\frac{2}{3}}QL}\sum_{l\in\mathscr{L}}\overline{\lambda_{\pi}(1,l)}l^{\frac{2}{3}}&\sum_{\substack{1\leq q\leq Q\\ (q,pl)=1}}\frac{1}{q^{5/2}}\sum_{n_1|pqr}n_1^{1/3}\sum_{p^{-3\epsilon}\frac{M_0}{n_1^2}\ll n_2\ll\frac{M_0}{n_1^2}}\frac{\lambda_{\pi}(n_1,n_2)}{n_2^{1/3}}\\
&\sum_{n\ll N_0}\frac{\lambda(n)}{n^{1/4}}C_2(\dots)I(n_1^2n_2,n,q),
\end{aligned}
\end{equation}where
\begin{equation}
\begin{aligned}
C_2(\dots)&=\sum_{u=1}^{p-1}\sideset{}{^*}\sum_{a (q)}S\left(r(\overline{ap+uq}),n_2,pqr/n_1\right)C_1(n\bar{l},a,q,u),\\
&=\sideset{}{^*}\sum_{\alpha (pqr/n_1)}f(\alpha,n\bar{l},q)\tilde{S}(\alpha,n\bar{l},q)e\left(\frac{\bar{\alpha}n_2n_1}{pqr}\right)\,,
\end{aligned}
\end{equation}where
\begin{equation*}
\begin{aligned}
&\tilde{S}(\alpha,n,q)=\sideset{}{^*}\sum_{u (p)}\sum_{\substack{b (p)\\ b\neq ul}}\bar{\chi}(b)e\left(\frac{\bar{q}^2(n_1\alpha\bar{u}+n\overline{(u-b\bar{l})})}{p}\right),\\
&f(\alpha,n,q)=\sum_{\substack{d|q\\n_1\alpha\equiv-n\bmod d}} d\mu\left(q/d\right)\,,
\end{aligned}
\end{equation*}and 
\begin{equation}\label{iinte}
\mathcal{I}(m,n,q)=\int\int\int g(q,x)V(z)U(y)e\left(\frac{lNx(z-y)}{pqQ}\pm\frac{2\sqrt{nNy}}{pq}\pm\frac{3(Nlmz)^{1/3}}{pqr^{1/3}}\right)dy\,dz\,dx.
\end{equation}
Note that the inverse $\overline{ap+uq}$ is taken modulo $pq$.
Splitting $q$ in dyadic blocks $q\sim C$ with 
\begin{equation}\label{cop}
q=q_1q_2, \,\,\,\,\,q_1|(n_1r)^{\infty}, (q_2,n_1r)=1,
\end{equation}
the $C$ block in \eqref{comb} is 
\begin{equation}\label{5.6}
\begin{aligned}
\ll \frac{N^{17/12}}{r^{2/3}p^3QC^{5/2}L}\sum_{n_1\ll Cpr}n_1^{1/3}&\sum_{\substack{\frac{n_1}{(n_1,pr)}|q_1|(n_1r)^{\infty}}}\sum_{p^{-3\epsilon}M_0/n_1^2\ll n_2\ll M_0/n_1^2}\frac{|\lambda_{\pi}(n_1,n_2)|}{n_2^{1/3}}\\
&\times \Bigg|\sum_{l\in\mathscr{L}}\overline{\lambda_{\pi}(1,l)}l^{\frac{2}{3}}\sum_{\substack{q_2\sim C/q_1\\ (q, pl)=1\\(q_2,rn_1)=1}}\sum_{n\ll N_0}\frac{\lambda(n)}{n^{1/4}}C_2(\dots)\,\,\mathcal{I}(n_1^2n_2,n,q )\Bigg|.
\end{aligned}
\end{equation} Applying the Cauchy-Schwarz inequality yields that \eqref{5.6} is at most
\begin{equation}\label{eq14}
\ll\frac{N^{17/12}}{r^{2/3}p^3QC^{5/2}L}\sup_{N_1\ll N_0}\sum_{n_1\ll Cpr}n_1^{1/3}\Theta^{1/2}\sum_{\frac{n_1}{(n_1,pr)}|q_1|(pn_1)^{\infty}}\Omega^{1/2}\,,
\end{equation}where
\begin{equation}
\Theta=\sum_{n_2\ll M_0/n_1^2}\frac{|\lambda_{\pi}(n_1,n_2)|^2}{n_2^{2/3}},
\end{equation}and
\begin{equation}\label{59}
\Omega=\sum_{p^{-3\epsilon}M_0/n_1^2\ll n_2\ll M_0/n_1^2}\Bigg|\sum_{l\in\mathscr{L}}\overline{\lambda_{\pi}(1,l)}l^{\frac{2}{3}}\sum_{\substack{q_2\sim C/q_1\\(q,pl)=1\\(q_2,rn_1)=1}}\sum_{n\sim N_1}\frac{\lambda(n)}{n^{1/4}}C_2(\dots)\,\,\mathcal{I}(n_1^2n_2,n,q)\Bigg|^2.
\end{equation}

\subsection{Poisson summation}
Smoothing out the $n_2$ sum in \eqref{59} with an appropriate bump function $W$, and opening the absolute value square, we arrive at
\begin{equation}
\begin{aligned}
\Omega &=\sum_{l,l'\sim L}\overline{\lambda_{\pi}(1,l)}\lambda_{\pi}(1,l')(ll')^{\frac{2}{3}}\sum_{n,n'\sim N_1}\frac{\lambda(n)\overline{\lambda(n')}}{(nn')^{1/4}}\sum_{\substack{q_2,q'_2\sim C/q_1\\(q,pl)=(q',pl')=1\\(q_2,rn_1)=(q_2',rn_1)=1)}}\\
&\qquad\qquad\qquad\times\sideset{}{^*}\sum_{\alpha (pqr/n_1)}\sideset{}{^*}\sum_{\alpha' (pq'r/n_1)}f(\alpha,n\bar{l},q)\tilde{S}(\alpha,n\bar{l},q)\bar{f}(\alpha ',n'\bar{l'},q')\bar{\tilde{S}}(\alpha ',n'\bar{l'},q')\\
&\qquad\qquad\qquad\times\sum_{n_2\in\mathbb{Z}}W\left(\frac{p^{3\epsilon}n_1^2n_2}{M_0}\right)e\left(n_2\left(\frac{n_1\overline{\alpha}}{pqr}-\frac{n_1\overline{\alpha '}}{pq'r}\right)\right)\mathcal{I}(n_1^2n_2,n,q)\overline{\mathcal{I}(n_1^2n_2,n',q')}\,,
\end{aligned}
\end{equation}where $q=q_1q_2$ and $q'=q_1q'_2$. We apply Poisson summation formula on the $n_2$ sum with modulus $pq_1q_2q_2'r/n_1$. After a change of variable in the Fourier transform, we arrive at
\begin{equation}\label{eq21}
\Omega\ll \frac{M_0L^{4/3}}{n_1^2N_1^{1/2}}\sum_{l,l'\sim N}|\overline{\lambda_{\pi}(1,l)}\lambda_{\pi}(1,l')|\sum_{n,n'\sim N_1}|\lambda(n)\overline{\lambda(n')}|\sum_{\substack{q_2,q'_2\sim C/q_1\\(q,pl)=1\\(q',pl')=1}}\sum_{n_2\in\mathbb{Z}}|\mathcal{C}||\mathcal{J}|\,,
\end{equation} where the character sum $\mathcal{C}$ is given by 
\begin{equation}\label{510}
\begin{aligned}
\mathcal{C}=&\sideset{}{^*}\sum_{u (p)}\sideset{}{^*}\sum_{u' (p)}\left(\sum_{\substack{b (p)\\b\neq ul}}\bar{\chi} (b)e\left(\frac{n\overline{q^2l(u-b\bar{l})}}{p}\right)\right)\left(\sum_{\substack{b' (p)\\b'\neq u'l'}}\chi (b')e\left(\frac{-n'\overline{{q'}^2l'(u'-b'\bar{l'})}}{p}\right)\right)\\
& \times \left(\sum_{d|q}\sum_{d'|q'}dd'\mu (q/d)\mu (q'/d')\sideset{}{^*}\sum_{\alpha (\frac{pqr}{n_1})}\sideset{}{^*}\sum_{\substack{\alpha ' (\frac{pq'r}{n_1})\\q_2'\bar{\alpha}-q_2\bar{\alpha'}\equiv n_2 (\frac{prq_2q_2'q_1}{n_1})\\n_1\alpha\equiv -n\bar{l}(d)\\n_1\alpha'\equiv -n'\bar{l'}(d')}}e\left(\frac{n_1\alpha\overline{uq^2}-n_1\alpha '\overline{u'{q'}^2}}{p}\right)\right)\,,
\end{aligned}
\end{equation} and the integral transform $\mathcal{J}$ is given by
\begin{equation}
\mathcal{J}=\int_{\mathbb{R}}W(w)\mathcal{I}(p^{-3\epsilon}M_0w,n,q)\overline{\mathcal{I}(p^{-\epsilon}M_0w,n',q')}e\left(-\frac{p^{-3\epsilon}M_0n_2w}{n_1prq_2q_2'q_1}\right)dw\,,
\end{equation}where $\mathcal{I}$ as in \eqref{iinte}. Note that the oscillation of $w$ in $\mathcal{I}(p^{-3\epsilon}M_0w,n,q)\overline{\mathcal{I}(p^{-3\epsilon}M_0w,n',q')}$ is of size
\begin{equation}\label{eo}
\frac{(p^{-3\epsilon}NM_0L)^{1/3}}{pCr^{1/3}}\gg p^{-2\epsilon/3}
\end{equation}
It follows by repeated integration by parts in the $w$ integral that $\mathcal{J}$ is negligibly small unless
\begin{equation}
\frac{p^{\epsilon}(p^{-3\epsilon}NM_0L)^{1/3}}{pCr^{1/3}}\gg \frac{p^{-3\epsilon}q_1M_0n_2}{n_1C^2r},
\end{equation}that is,
\begin{equation}\label{515}
n_2\ll \frac{p^{3\epsilon}n_1r^{2/3}N^{1/3}L^{1/3}C}{q_1pM_0^{2/3}}\ll \frac{p^{3\epsilon}n_1L^{1/2}}{q_1}\left(\frac{N}{p^3}\right)^{1/2}:= N_2\,.
\end{equation}
\begin{remark}\label{Remark 4}
In the case \eqref{comp}, the extra oscillation \eqref{eo} is not present and we get a smaller dual range
\begin{equation}
p^{\epsilon}\gg \frac{q_1\tilde{M}_0n_2}{n_1C^2r},
\end{equation}where $\tilde{M}_0$ as in \eqref{comp}.
\end{remark}
\noindent
For smaller values of $q$, there is oscillation in the integrand of $\mathcal{I}$ and hence we have the following bound
\begin{lemma}\label{int}
For $w\asymp 1$ we have,
\begin{equation}\label{inte}
\mathcal{I}(M_0w,n,q)\ll \frac{p^{1+\epsilon}qQ}{NL}\times\sqrt{\frac{pqr^{1/3}}{(NM_0l)^{1/3}}}= p^{5\epsilon/6}\left(\frac{q}{Q}\right)^{3/2}\,.
\end{equation}and consequently,
\begin{equation}
\mathcal{J}\ll p^{\epsilon}\left(\frac{q}{Q}\right)^3\,.
\end{equation}
\end{lemma}	
\begin{proof}Note that if $q>p^{-\epsilon/3}Q$, then we obtain \eqref{inte} trivially by taking absolute values inside $\mathcal{I}$. Hence we assume $Q/q\geq p^{\epsilon/3}$.

Changing variable $u=z-y$, we get
\begin{equation}\label{5.14}
\mathcal{I}=\int\int\int g(q,x) V(u+y)U(y)e\left(\frac{lNxu}{pqQ}\pm\frac{2\sqrt{nNy}}{pq}\pm\frac{3(NlM_0w(y+u))^{1/3}}{pqr^{1/3}}\right)dy\,du\,dx\,.
\end{equation}  Suppose first $\min \{Q/q,1/|x|\}\leq p^{\epsilon/8}$, that is $x\geq p^{-\epsilon/8}$. Then by the third property in \eqref{delta}, we have $x^jg^{(j)}(q,x)\ll p^{\epsilon/7}$, and hence from the $x$ integral, it follows by repeated integration by parts (after introducing a dyadic partition of unity to the $x$ integral) that 
\begin{equation}\label{u}
u\ll p^{\epsilon/6} (pqQ)/(NL)\ll p^{-\epsilon/6}\,,
\end{equation}
for non-negligible contribution. Now consider the $y$ integral. Changing variable $y\mapsto y^2$, the $y$ integral becomes
\begin{equation}
2\int y V(y^2+u)U(y^2)e(\phi(y,u)) \,dy\,,
\end{equation}where
\begin{equation}
\phi(y,u)=\pm\frac{2\sqrt{nN}y}{pq}\pm\frac{3(NM_0lw(y^2+u))^{1/3}}{pqr^{1/3}}\,.
\end{equation}We have
\begin{equation}\label{s}
\frac{\partial^2\phi(y,u)}{\partial y^2}=\pm\frac{2(NM_0lw)^{1/3}(3u-y^2)}{3pqr^{1/3}(y^2+u)^{5/3}}\gg \frac{(NM_0L)^{1/3}}{pqr^{1/3}}\,,
\end{equation}where the last inequality occurs since $u\ll p^{-\epsilon/6}$. Hence the Lemma follows by \eqref{u} and the second derivative bound for the $y$ integral.
On the other hand, if $\min \{Q/q,1/|x|\}>p^{\epsilon/8}$, then by the first property of \eqref{delta}, we can replace $g$ by $1$. The $x$ integral then saves a factor $pqQ/Nl$ and we are left with the $(u,y)$ integral
\begin{equation}\label{5.24}
\int \int\frac{V(y+u)}{u}e\left(\frac{2\sqrt{nNy}}{pq}\pm\frac{3(NM_0lw(y+u))^{1/3}}{pqr^{1/3}}\right) \,du\,dy.
\end{equation}Now if $u\leq p^{-\epsilon}$, then following the same argument from \eqref{u}-\eqref{s} we get the required square-root factor from the $y$ integral. If $u>p^{-\epsilon}$, then the Lemma follows after applying the second derivative bound for the $u$ integral in \eqref{5.24}.
\end{proof}

It remains to estimate the character sum $\mathcal{C}$. As usual, the cases $n_2=0\bmod p$ and $n_2\neq 0\bmod p$ have to be dealt separately.
\section{The non-zero frequencies modulo $p$ $(n_2\neq 0\bmod p)$}
Note that it follows from \eqref{eq2.8}, \eqref{210} and \eqref{cop} that $(p,rq_2q_2'q_1/n_1)=(q_2q_2',prq_1/n_1)=1$, so that 
\begin{equation}
\mathcal{C}\ll |\mathcal{C}_1\mathcal{C}_2\mathcal{C}_3|, 
\end{equation}where
\begin{equation}\label{eq32}
\begin{aligned}
\mathcal{C}_1=\sideset{}{^*}\sum_{u (p)}\sideset{}{^*}\sum_{u' (p)}\left(\sum_{\substack{b (p)\\b\neq ul}}\bar{\chi} (b)e\left(\frac{n\overline{q^2l(u-b\bar{l})}}{p}\right)\right)\left(\sum_{\substack{b' (p)\\b'\neq u'l'}}\chi (b')e\left(\frac{-n'\overline{{q'}^2l'(u'-b'\bar{l'})}}{p}\right)\right)\\
\times\sideset{}{^*}\sum_{\substack{\alpha,\alpha' (p)\\q_2'\bar{\alpha}-q_2\bar{\alpha'}\equiv n_2 (p)}}e\left(\frac{n_1\alpha\overline{uq^2}-n_1\alpha '\overline{u'{q'}^2}}{p}\right)\,,
\end{aligned}
\end{equation}
\begin{equation}\label{eq33}
\mathcal{C}_2=\sum_{d_1|q_1}\sum_{d_1'|q_1}d_1d_1'\mathop{\sideset{}{^*}\sum_{\substack{\alpha(\frac{rq_1}{n_1})\\n_1\alpha=-n\bar{l}(d_1)}}\,\,\,\,\,\sideset{}{^*}\sum_{\substack{\alpha'(\frac{rq_1}{n_1})\\n_1\alpha'=-n'\bar{l'}(d_1')}}}_{q_2'\bar{\alpha}-q_2\bar{\alpha'}=n_2(\frac{rq_1}{n_1})}1\,,
\end{equation}and
\begin{equation}\label{eq34}
\mathcal{C}_3=\mathop{\sum\sum}_{\substack{d_2|q_2\\d_2'|q_2'}}d_2d_2'\mathop{\sideset{}{^*}\sum_{\substack{\alpha(q_2)\\n_1\alpha=-n\bar{l}(d_2)}}\,\,\,\,\,\sideset{}{^*}\sum_{\substack{\alpha'(q_2')\\n_1\alpha'=-n'\bar{l'}(d_2')}}}_{q_2'\bar{\alpha}-q_2\bar{\alpha'}=n_2(q_2q_2')}1\,.
\end{equation}We proceed for the estimation of $\mathcal{C}_1,\mathcal{C}_2$ and $\mathcal{C}_3$ according as $n_1=0\bmod p$ or not. 
\subsection{($n_1\neq 0\bmod p$)}We begin with the following estimate of $\mathcal{C}_1$, which is proved by Will Sawin in the appendix to this paper.
\begin{lemma}\label{c1}
For $n_1,n_2\neq 0\bmod p$, we have
\begin{equation}
\mathcal{C}_1\ll p^{5/2}.
\end{equation}
\end{lemma}	
\begin{proof}
We begin with some elementary transformations. Changing  $b$ to $blu$, the $(b,u)$ sum in \eqref{eq32} becomes
\begin{equation}\label{ub}
\begin{aligned}
&\sideset{}{^*}\sum_{u (p)}\sum_{\substack{b (p)\\b\neq ul}}\bar{\chi} (b)e\left(\frac{n\overline{q^2l(u-b\bar{l})}+n_1\alpha\overline{uq^2}}{p}\right)\\
&=\bar{\chi}(l)\sum_{\substack{b (p)\\b\neq 1}}\bar{\chi}(b)\sideset{}{^*}\sum_{u (p)}\bar{\chi}(u)e\left(\frac{\overline{u}\left(n\overline{q^2l(1-b)}+n_1\alpha\overline{q^2}\right)}{p}\right)\\
&=\bar{\chi}(l)g_{\chi}\sum_{\substack{b (p)\\b\neq 1}}\bar{\chi}(b)\bar{\chi}\big(n\overline{q^2l(1-b)}+n_1\alpha\overline{q^2}\big)\\
&=\bar{\chi}(l)\cdot g_{\chi}\cdot\bar{\chi}(n_1\overline{q^2})\sideset{}{^*}\sum_{\substack{b (p)\\}}\bar{\chi}(1-b)\bar{\chi}(n\overline{n_1l}\cdot\bar{b}+\alpha ).
\end{aligned}
\end{equation}Similarly, changing $b'$ to $b'l'u'$, the $(b',u')$ sum in \eqref{eq32} becomes
\begin{equation}\label{ub2}
\begin{aligned}
&\sideset{}{^*}\sum_{u' (p)}\sum_{\substack{b' (p)\\b'\neq u'l'}}\chi (b')e\left(\frac{-n'\overline{q'^2l'(u'-b'\bar{l'})}-n_1\alpha'\overline{u'q'^2}}{p}\right)\\
&=-\chi(l')g_{\bar{\chi}}\cdot\chi(n_1\overline{q'^2})\sideset{}{^*}\sum_{b' (p)}\chi(1-b')\chi(
n'\overline{n_1l'}\cdot\bar{b'}+\alpha').
\end{aligned}
\end{equation}Substituting the above two simplifications into \eqref{eq32}, and changing $(\alpha,\alpha')\mapsto (\overline{\alpha}, \overline{\alpha'})$ and using the congruence condition between $\alpha, \alpha'$, we get
\begin{equation}\notag
\mathcal{C}_1\ll p \Bigg|\sideset{}{^*}\sum_{\substack{  \alpha(p)  \\ \alpha \neq - k_2/k_1}} \sideset{}{^*}\sum_{b, b'(p)} \bar{\chi}(b+1) \bar{\chi}(c_1 \overline{b} + \overline{\alpha} ) \chi (b'+1) \chi ( c_2 \overline{b'} + \overline{k_1\alpha + k_2})\Bigg|,
\end{equation}with $c_1=-n\overline{n_1l}, c_2=-n'\overline{n_1l'}, k_1=\overline{q_2}q_2'$ and $k_2=-\overline{q_2}n_2$. Since $k_2\neq 0\bmod p$ by our assumption, the lemma follows from Theorem \ref{appendix-main-thm} in the appendix.
\end{proof}
\begin{lemma}\label{c2}We have
\begin{equation}\label{pr1}
\mathcal{C}_2\ll d(q_1)q_1\sum_{d_1'|q_1}d_1'\sum_{\substack{\alpha' \left(\frac{rq_1}{n_1}\right)\\n_1\alpha'=-n'\bar{l'} (d_1')}} 1.
\end{equation}Symmetrically, the same bound holds when $(q_2',l',n',\alpha')$ is replaced by $(q_2,l,n,\alpha)$.
\end{lemma}	
\begin{proof}Recall
\begin{equation}
\mathcal{C}_2=\sum_{d_1|q_1}\sum_{d_1'|q_1}d_1d_1'\mathop{\sideset{}{^*}\sum_{\substack{\alpha(\frac{rq_1}{n_1})\\n_1\alpha=-n\bar{l}(d_1)}}\,\,\,\,\,\sideset{}{^*}\sum_{\substack{\alpha'(\frac{rq_1}{n_1})\\n_1\alpha'=-n'\bar{l'}(d_1')}}}_{q_2'\bar{\alpha}-q_2\bar{\alpha'}=n_2(\frac{rq_1}{n_1})}1\,.
\end{equation}
Also recall the assumptions $(q_2,n_1r)=(q_2',n_1r)=1$ and $q_1|(n_1r)^{\infty}$ from \eqref{cop}. Hence, given $\alpha'$, there can be at most one $\alpha$. Hence we get
\begin{equation}\label{pr1}
\mathcal{C}_2\ll\sum_{d_1|q_1}d_1\sum_{d_1'|q_1}d_1'\sum_{\substack{\alpha' \left(\frac{rq_1}{n_1}\right)\\n_1\alpha'=-n'\bar{l'} (d_1')}} 1. \ll d(q_1)q_1\sum_{d_1'|q_1}d_1'\sum_{\substack{\alpha' \left(\frac{rq_1}{n_1}\right)\\n_1\alpha'=-n'\bar{l'} (d_1')}} 1.
\end{equation}
\end{proof}
\begin{lemma}\label{c3}We have
\begin{equation}
\mathcal{C}_3\ll d(q_2)(q_2,q_2'n_1l+nn_2)\sum_{d_2'|q_2'}d_2'\mathop{\sum_{\alpha(q_2)}\,\,\,\,\,\sum_{\substack{\alpha'(q_2')\\n_1\alpha'=-n'\bar{l'}(d_2')}}}_{q_2'\bar{\alpha}-q_2\bar{\alpha'}=n_2(q_2q_2')}1.
\end{equation}
\end{lemma}	
\begin{proof}
We have
\begin{equation}
\mathcal{C}_3=\mathop{\sum\sum}_{\substack{d_2|q_2\\d_2'|q_2'}}d_2d_2'\mathop{\sideset{}{^*}\sum_{\substack{\alpha(q_2)\\n_1\alpha=-n\bar{l}(d_2)}}\,\,\,\,\,\sideset{}{^*}\sum_{\substack{\alpha'(q_2')\\n_1\alpha'=-n'\bar{l'}(d_2')}}}_{q_2'\bar{\alpha}-q_2\bar{\alpha'}=n_2(q_2q_2')}1\,.
\end{equation}
Recall $(n_1,q_2q_2')=1$ from \eqref{cop}. Hence we get $\alpha=-n\overline{ln_1}\bmod d_2$. Note that it is implicit in congruence conditions that $(n,d_2)=(n',d_2')=1$.  Then using the congruence relation modulo $q_2q_2'$, we obtain $-q_2'n_1l\bar{n}=n_2\bmod d_2$.
Hence
\begin{equation}\label{pr2}
\begin{aligned}
\mathcal{C}_3&\ll\sum_{d_2|(q_2,q_2'n_1l+nn_2)}d_2\sum_{d_2'|q_2'}d_2'\mathop{\sum_{\alpha(q_2)}\,\,\,\,\,\sum_{\substack{\alpha'(q_2')\\n_1\alpha'=-n'\bar{l'}(d_2')}}}_{q_2'\bar{\alpha}-q_2\bar{\alpha'}=n_2(q_2q_2')}1\\
&\ll d(q_2)(q_2,q_2'n_1l+nn_2)\sum_{d_2'|q_2'}d_2'\mathop{\sum_{\alpha(q_2)}\,\,\,\,\,\sum_{\substack{\alpha'(q_2')\\n_1\alpha'=-n'\bar{l'}(d_2')}}}_{q_2'\bar{\alpha}-q_2\bar{\alpha'}=n_2(q_2q_2')}1.
\end{aligned}
\end{equation}
\end{proof}
\noindent
Combining Lemma \ref{c1}, Lemma \ref{c2} and Lemma \ref{c3} we obtain,
\begin{lemma}\label{char}
For $n_1,n_2\neq 0\bmod p$, we have
\begin{equation}
\mathcal{C}\ll
p^{5/2}q_1(q_2,q_2'n_1l+nn_2)\mathop{\sum_{\alpha(q_2)}\,\,\,\,\,\sum_{\alpha'(q_2')}}_{q_2'\bar{\alpha}-q_2\bar{\alpha'}=n_2(q_2q_2')}\sum_{\beta' \left(\frac{rq_1}{n_1}\right)}\mathop{\sum_{d_2'|q_2'}\sum_{d_1'|q_1}}_{\substack{n_1\alpha'=-n'\bar{l'} (d_2')\\n_1\beta'=-n'\bar{l'}(d_1')}}d_2'd_1'\,.
\end{equation}
\end{lemma}Note that similar bounds holds when the roles of $(q_2',l',n',\alpha')$ and $(q_2,l,n,\alpha)$ are interchanged.

\subsection{$n_1=0\,(\bmod\, p)$} In this case, the last exponential factor in \eqref{eq32} is not present and consequently the character sum $\mathcal{C}_1$ is much simpler. Indeed, we now have
\begin{equation}\label{620}
\begin{aligned}
\mathcal{C}_1&=p\sideset{}{^*}\sum_{u (p)}\sideset{}{^*}\sum_{u' (p)}\left(\sum_{\substack{b (p)\\ b\neq ul}}\bar{\chi} (b)e\left(\frac{n\overline{q^2l(u-b\bar{l})}}{p}\right)\right)\left(\sum_{\substack{b' (p)\\ b'\neq u'l'}}\chi (b')e\left(\frac{-n'\overline{{q'}^2l'(u'-b'\bar{l'})}}{p}\right)\right)\\
&=p\sideset{}{^*}\sum_{b,b' (p)}\bar{\chi}(b)\chi(b')\sum_{\substack{k,k' (p)\\k\neq 0,-\bar{b}l\\k'\neq 0,-\bar{b'}l'}}e\left(\frac{nk\overline{q^2l}}{p}-\frac{n'k'\overline{q'^2l'}}{p}\right)\\
&=p\sideset{}{^*}\sum_{b,b' (p)}\bar{\chi}(b)\chi(b')e\left(\frac{nb\overline{q^2}}{p}-\frac{n'b'\overline{q'^2}}{p}\right)\ll p^2.
\end{aligned}
\end{equation}The estimation of $\mathcal{C}_2$ and $\mathcal{C}_3$ remains the same as  \eqref{pr1} and \eqref{pr2} respectively. Compared to the earlier case, we save an extra $p^{1/2}$ in the $\mathcal{C}_1$ sum, and consequently, we save more in this case.
\\
\\
Let $\Omega_{\neq 0\bmod p}$ denote the contribution of $n_2\neq 0\bmod p$ to $\Omega$ in \eqref{eq21}, and let $\sum_{\neq 0\bmod p}$ be its contribution to \eqref{eq14}.
\begin{lemma}\label{dd}
We have
\begin{equation}\label{sl}
\begin{aligned}
\Omega_{\neq 0\bmod p}&\ll\frac{p^{5/2}C^3M_0rC^2N_2L^{2+4/3}N_1}{n_1^3Q^3N_1^{1/2}}\left(C+N_1\right),\\
&\ll \frac{r^2L^{3+4/3}p^{11/2}C^5}{n_1^2q_1N^{1/2}}\,.
\end{aligned}
\end{equation}and
\begin{equation}\label{nd}
\sum_{\neq 0\bmod p}\ll N^{3/4}p^{1/2}L^{3/4}r^{1/2}.
\end{equation}
\end{lemma}	
\begin{proof}
We work out the bounds for $n_1\neq 0\bmod p$, since the other case contributes less as pointed out earlier. We substitute the bound from Lemma \ref{char} in \eqref{eq21}. We also use the inequality
\begin{equation}\label{amgm}
|\lambda(n)\overline{\lambda(n')}|\ll |\lambda(n)|^2+|\lambda(n')|^2
\end{equation}and consider the contribution of only the first term $|\lambda(n)|^2$, since the calculation for the second term is absolutely similar. Executing the above steps, we see that the contribution of the case under consideration towards $\Omega_{\neq 0\bmod p}$ is
\begin{equation}\label{eq38}
\begin{aligned}
 &\ll\frac{p^{5/2}M_0L^{4/3}}{n_1^2N_1^{1/2}}\mathop{\sum\sum}_{l,l'\sim L}|\overline{\lambda_{\pi}(1,l)}\lambda_{\pi}(1,l')|\,\,\sum_{\substack{q_2,q'_2\sim C/q_1\\(q,pl)=1\\(q',pl')=1}}\sum_{n_2\ll N_2}\sum_{n\sim N_1}|\lambda_f(n)|^2q_1(q_2,q_2'n_1l+nn_2)\\
&\quad\quad\quad\quad\quad\quad\quad\quad\quad\quad\quad\times\mathop{\sum_{\alpha(q_2)}\,\,\,\,\,\sum_{\alpha'(q_2')}}_{q_2'\bar{\alpha}-q_2\bar{\alpha'}=n_2(q_2q_2')}\sum_{\beta' \left(\frac{rq_1}{n_1}\right)} \sum_{d_2'|q_2'}\sum_{d_1'|q_1}d_2'd_1' \sum_{\substack{n'\sim N_1\\n_1\alpha'=-n'\bar{l'} (d_2')\\n_1\beta'=-n'\bar{l'}(d_1')}}|\mathcal{J}|.
\end{aligned}
\end{equation}
We substitute the bound for $\mathcal{J}\ll (q/Q)^3$ from Lemma \eqref{int}. We next execute the sum over $n'$ to get
\begin{equation}\label{627}
\begin{aligned}
 \sum_{d_2'|q_2'}\sum_{d_1'|q_1}d_2'd_1' \sum_{\substack{n'\sim N_1\\n_1\alpha'=-n'\bar{l'} (d_2')\\n_1\beta'=-n'\bar{l'}(d_1')}}&\ll  \sum_{d_2'|q_2'}\sum_{d_1'|q_1}d_2'd_1' \left(1+\frac{(d_1',d_2')N_1)}{d_1'd_2'}\right)\,,\\
 &\ll (q_1q_2'+(q_1,q_2')N_1)\,.
 \end{aligned}
\end{equation}The sum over $\alpha , \alpha' $ and $\beta' $ in \eqref{eq38} is now bounded by $(q_2,n_2)rq_1/n_1$. Substituting, we get that \eqref{eq38} is
\begin{equation}\label{6.27}
\begin{aligned}
&\ll\frac{p^{5/2}C^3M_0L^{4/3}q_1^2r}{n_1^3Q^3N_1^{1/2}}\mathop{\sum\sum}_{l,l'\sim L}|\overline{\lambda_{\pi}(1,l)}\lambda_{\pi}(1,l')|\\
&\times\sum_{q_2'\sim C/q_1}\sum_{n_2\ll N_2}\sum_{n\sim N_1}|\lambda(n)|^2(q_1q_2'+(q_1,q_2')N_1)\sum_{q_2\sim C/q_1}(q_2,n_2)(q_2,q_2'n_1l+nn_2)\,.
\end{aligned}
\end{equation}
The over $q_2$ is then bounded by $(n_2,q_2'n_1l)C/q_1$, which summed over $n_2$ is bounded by $N_2C/q_1$. Hence we are left with
\begin{equation}\label{628}
\begin{aligned}
&\frac{p^{5/2}C^3M_0L^{4/3}q_1rCN_2N_1^{7/32}}{n_1^3Q^3N_1^{1/2}}\mathop{\sum\sum}_{l,l'\sim L}|\overline{\lambda_{\pi}(1,l)}\lambda_{\pi}(1,l')|\,\,\sum_{q_2'\sim C/q_1}\sum_{n\sim N_1}|\lambda(n)|^2(q_1q_2'+(q_1,q_2')N_1)\,.
\end{aligned}
\end{equation}
Executing the remaining sum we get
\begin{equation}\label{6.29}
\frac{p^{5/2}C^3M_0q_1rC^2N_2L^{2+4/3}N_1}{n_1^3Q^3N_1^{1/2}}\left(\frac{C}{q_1}+\frac{N_1}{q_1}\right)\,,
\end{equation}
where we have used the Ramanujan bound on average \eqref{rama} for $l,l'$ and $n$ sum. This is the first line of the lemma. Note that after substituting $N_1\ll N_0$, the contribution from $N_1$ in the parenthesis dominates the other. Substituting $N_2=(p^{\epsilon}n_1L^{1/2}N^{1/2})/(q_1p^{3/2})$, $N_1\ll N_0=p^{1+\epsilon}L, M_0=rp^{3/2+\epsilon}N^{1/2}L^{1/2}$ and $Q=\sqrt{(NL)/p}$, we get the second line of the first part of the lemma.

Using the bound from the second line of \eqref{sl} in \eqref{eq14}, we get
\begin{equation}\label{630}
\begin{aligned}
&\sum_{\neq 0\bmod p}\\
&\ll \frac{N^{17/12}}{r^{2/3}p^3QC^{5/2}L}\left(\frac{r^2L^{3+4/3}p^{11/2}C^5}{N^{1/2}}\right)^{1/2}\sum_{n_1\ll Cpr}\frac{n_1^{1/3}\Theta^{1/2}}{n_1}\sum_{\frac{n_1}{(n_1,pr)}|q_1|(n_1r)^{\infty}}\frac{1}{q_1^{1/2}}\\
&\ll N^{2/3}p^{1/4}L^{2/3}r^{1/3}\sum_{n_1}\frac{(n_1,pr)^{1/2}\Theta^{1/2}}{n_1^{7/6}}\,,
\end{aligned}
\end{equation}where 
\begin{equation}
\Theta=\sum_{n_2\ll M_0/n_1^2}\frac{|\lambda_{\pi}(n_1,n_2)|^2}{n_2^{2/3}}\,.
\end{equation}
An application of Cauchy-Schwarz gives
\begin{equation}\label{632}
\sum_{n_1\ll Cpr}\frac{(n_1,pr)^{1/2}\Theta^{1/2}}{n_1^{7/6}}\ll \left(\sum_{n_1\ll pQr}\frac{(n_1,pr)}{n_1}\right)^{1/2}\left(\sum_{n_1^2n_2\ll M_0}\frac{|\lambda_{\pi}(n_1,n_2)|^2}{(n_1^2n_2)^{2/3}}\right)^{1/2}\ll M_0^{1/6}\,,
\end{equation}where in the last sum we applied the Ramanujan bound on average \eqref{rama} and partial summation. The second part of the lemma follows after substituting \eqref{632} in \eqref{630}.
\end{proof}
\section{ The zero frequencies modulo $p$}\label{zerofreq}We proceed to obtain bounds for the sum character sum $\mathcal{C}$ in \eqref{510} for $n_2=0\bmod p $. First, we note that from the congruence relation
\begin{equation}
q_2'\bar{\alpha}-q_2\bar{\alpha'}\equiv n_2 (\frac{prq_2q_2'q_1}{n_1})
\end{equation}in $\mathcal{C}$ \eqref{510}, $n_2=0\bmod p$ implies
\begin{equation}\label{7.4}
\alpha'=\overline{q'_2}q_2\alpha\bmod p.
\end{equation}Substituting \eqref{7.4} in $\mathcal{C}$, and executing the sum over $\alpha\bmod p$ part, we get
\begin{equation}
\begin{aligned}
\mathcal{C}=&\sideset{}{^*}\sum_{u (p)}\sideset{}{^*}\sum_{u' (p)}\left(\sum_{\substack{b (p)\\b\neq ul}}\bar{\chi} (b)e\left(\frac{n\overline{q^2l(u-b\bar{l})}}{p}\right)\right)\left(\sum_{\substack{b' (p)\\b'\neq u'l'}}\chi (b')e\left(\frac{-n'\overline{{q'}^2l'(u'-b'\bar{l'})}}{p}\right)\right)\\
& \times \left(\sum_{d|q}\sum_{d'|q'}dd'\mu (q/d)\mu (q'/d')\sideset{}{^*}\sum_{\alpha (\frac{qr}{n_1})}\sideset{}{^*}\sum_{\substack{\alpha ' (\frac{q'r}{n_1})\\q_2'\bar{\alpha}-q_2\bar{\alpha'}\equiv n_2 (\frac{rq_2q_2'q_1}{n_1})\\n_1\alpha\equiv -n\bar{l}(d)\\n_1\alpha'\equiv -n'\bar{l'}(d')}}\left(p\delta_{uq_2^3=u'q_2'^3}-1\right)\right)\,,
\end{aligned}
\end{equation}Rearranging we get
\begin{equation}\label{7.9}
\mathcal{C}\leq \left(p|\mathcal{D}_0|+|\mathcal{D}_1|\right)|\mathcal{D}_2|
\end{equation}where
\begin{equation}\label{d0d1}
\begin{aligned}
&\mathcal{D}_0=\sideset{}{^*}\sum_{u (p)}\left(\sum_{\substack{b (p)\\b\neq ul}}\bar{\chi} (b)e\left(\frac{n\overline{q^2l(u-b\bar{l})}}{p}\right)\right)\left(\sum_{\substack{b' (p)\\ b'\neq q_2^3\overline{q_2'}^3ul'}}\chi (b')e\left(\frac{-n'\overline{{q'}^2l'(q_2^3\overline{q'_2}^3u-b'\bar{l'})}}{p}\right)\right)\,,\\
&\mathcal{D}_1=\sideset{}{^*}\sum_{u (p)}\sideset{}{^*}\sum_{u' (p)}\left(\sum_{\substack{b (p)\\b\neq ul}}\bar{\chi} (b)e\left(\frac{n\overline{q^2l(u-b\bar{l})}}{p}\right)\right)\left(\sum_{\substack{b' (p)\\b'\neq u'l'}}\chi (b')e\left(\frac{-n'\overline{{q'}^2l'(u'-b'\bar{l'})}}{p}\right)\right),
\end{aligned}
\end{equation} and
\begin{equation}
\mathcal{D}_2=\sum_{d|q}\sum_{d'|q'}dd'\sideset{}{^*}\sum_{\alpha (\frac{qr}{n_1})}\sideset{}{^*}\sum_{\substack{\alpha ' (\frac{q'r}{n_1})\\q_2'\bar{\alpha}-q_2\bar{\alpha'}\equiv n_2 (\frac{rq_2q_2'q_1}{n_1})\\n_1\alpha\equiv -n\bar{l}(d)\\n_1\alpha'\equiv -n'\bar{l'}(d')}}1.
\end{equation}We obtain the bounds for the character sums $\mathcal{D}_0,\mathcal{D}_1$ in the following two lemmas. For $\mathcal{D}_0$, we consider the following general sum which we will frequently encounter in this section.
\begin{lemma}\label{charb}
For $c_1,c_2,\beta \in \mathbb{Z}$, define
\begin{equation}\label{71}
\tilde{\mathcal{C}}:=\sideset{}{^*}\sum_{u\bmod p}\,\,\sideset{}{^*}\sum_{\substack{b\bmod p\\b\neq u}}\,\,\sideset{}{^*}\sum_{\substack{b'\bmod p\\b'\neq\beta u}}\bar{\chi}(b)\chi(b')e\left(\frac{c_1\overline{u-b}-c_2\overline{\beta u-b'}}{p}\right).
\end{equation}We have
\begin{equation}
\tilde{\mathcal{C}}\ll
\begin{cases}
p^{2}, & \quad \beta c_1=c_2\bmod p,\\
p, & \quad \text{otherwise}.
\end{cases}
\end{equation}
\end{lemma}	
\begin{proof}
Changing variables $\overline{u-b}\mapsto b_1$ and $\overline{u'-b'}\mapsto b_2$, we obtain
\begin{equation}\label{zerocharactersum}
\begin{aligned}
\tilde{\mathcal{C}}&=\sideset{}{^*}\sum_{u (p)}\,\,\sideset{}{^*}\sum_{\substack{b_1 (p)\\}}\,\,\sideset{}{^*}\sum_{\substack{b_2 (p)\\}}\bar{\chi}(u-\bar{b}_1)\chi(\beta u-\bar{b}_2)e\left(\frac{c_1b_1-c_2b_2}{p}\right)\\
&=\sideset{}{^*}\sum_{u (p)}\,\,\sideset{}{^*}\sum_{\substack{b_1 (p)\\}}\,\,\sideset{}{^*}\sum_{\substack{b_2 (p)\\}}\bar{\chi}(b_1b_2u-b_2)\chi(b_1b_2\beta u-b_1)e\left(\frac{c_1b_1-c_2b_2}{p}\right)\\
&=\sideset{}{^*}\sum_{u (p)}\,\,\sideset{}{^*}\sum_{\substack{b_1 (p)\\}}\,\,\sideset{}{^*}\sum_{\substack{b_2 (p)\\}}\bar{\chi}(u-b_2)\chi(\beta u-b_1)e\left(\frac{c_1b_1-c_2b_2}{p}\right)\\
&=\sideset{}{^*}\sum_{u (p)}\sideset{}{^*}\sum_{\substack{b_1 (p)\\}}e\left(\frac{\beta c_1 u}{p}\right)\chi(\beta u-b_1)e\left(\frac{-c_1(\beta u-b_1)}{p}\right)\\
&\quad\quad\quad\quad\quad\quad\quad\quad\quad\quad\quad\quad\quad\sideset{}{^*}\sum_{\substack{b_2 (p)\\}}e\left(\frac{-c_2u}{p}\right)\bar{\chi}(u-b_2)e\left(\frac{c_2(u-b_2)}{p}\right)
\end{aligned}
\end{equation}First suppose $c_1=c_2=0\bmod p$. Then the $b_1$ and $b_2$ sum above is $-\chi(\beta u)$ and $-\bar{\chi}(u)$ respectively. Multiplying and summing over $u$, we get the desired bound in this case.

Next suppose $c_1=0\bmod p$ and $c_2\neq 0\bmod p$. Then the $b_1$ sum above evaluates to $-\chi(\beta u)$ and the $b_2$ sum is $\chi(c_2)g_{\bar{\chi}}e(-c_2u/p)-\bar{\chi}(u)$. Hence
\begin{equation}\notag
\tilde{\mathcal{C}}=\sideset{}{^*}\sum_{u (p)}(-g_{\bar{\chi}}\chi(c_2)\chi(\beta u)e(-c_2u/p)+\bar{\chi}(u)\chi(\beta u))\ll p,
\end{equation}in this case as well. Similar arguments holds for the $c_2=0 \bmod p$ and $c_1\neq 0\bmod p$ case.

Now suppose $c_1,c_2\neq 0\bmod p$, but $\beta =0\bmod p$. Then the $b_1$ sum is the Gauss sum $\bar{\chi}(-c_1)g_{\chi}$ and the $b_2$ sum is $\chi(c_2)g_{\bar{\chi}}e(-c_2u/p)-\bar{\chi}(u)$, and we get
\begin{equation}\notag
\tilde{\mathcal{C}}=\bar{\chi}(-c_1)g_{\chi}\sideset{}{^*}\sum_{u (p)}(g_{\bar{\chi}}\chi(c_2)e(-c_2u/p)-\bar{\chi}(u))\ll p.
\end{equation}

Hence for the rest of the argument, we assume $c_1,c_2,\beta\neq 0\bmod p$. Evaluating the last $b_1,b_2$ sum in \eqref{zerocharactersum}, we obtain,
\begin{equation}\notag
\begin{aligned}
\tilde{\mathcal{C}}&=\bar{\chi}(-c_1)\chi(c_2)\sideset{}{^*}\sum_{u (p)}\left(g_{\chi}e\left(\frac{\beta c_1u}{p}\right)-\chi (-\beta c_1u)\right)\left(g_{\bar{\chi}}e\left(\frac{-c_2u}{p}\right)-\bar{\chi}(c_2u)\right)\\
&=\bar{\chi}(-c_1)\chi(c_2)\sideset{}{^*}\sum_{u (p)}\Bigg(p\cdot e\left(\frac{(\beta c_1-c_2)u}{p}\right)-g_{\chi}\cdot\bar{\chi}(c_2u)e\left(\frac{\beta c_1u}{p}\right)\\
 &\quad\quad\quad\quad\quad\quad\quad\quad\quad\quad\quad\quad-g_{\bar{\chi}}\cdot \chi(-\beta c_1u)e\left(\frac{-c_2u}{p}\right)+\chi(-\beta c_1u)\bar{\chi}(c_2u)\Bigg)\\
 &\ll p^2\delta_{\beta c_1=c_2 \bmod p}+p,
\end{aligned}
\end{equation}and the lemma follows.
\end{proof}
\begin{lemma}\label{•}
With $\mathcal{D}_1$ as in \eqref{d0d1}, we have
\begin{equation}\label{7.12}
\mathcal{D}_1\ll  p.
\end{equation}
\end{lemma}	
\begin{proof}
Consider the $(b,u)$ sum in $\mathcal{D}_1$. Executing the $u$ sum first we obtain
\begin{equation}\notag
\sideset{}{^*}\sum_{\substack{u (p)\\u\neq b\bar{l}}}e\left(\frac{n\overline{q^2l(u-b\bar{l})}}{p}\right)=p\delta_{n=0 (p)}-1-e\left(\frac{-n\overline{q^2b}}{p}\right).
\end{equation}Substituting, the $(b,u)$ sum in $\mathcal{D}_1$ \eqref{d0d1} becomes
\begin{equation}\notag
\sum_{b(p)}\bar{\chi}(b)\sideset{}{^*}\sum_{\substack{u (p)\\u\neq b\bar{l}}}e\left(\frac{n\overline{q^2l(u-b\bar{l})}}{p}\right)=-\bar{\chi}(-n\overline{q^2})g_{\chi}\ll p^{1/2}.
\end{equation}One similarly obtains that the $(b',u')$ sum in $\mathcal{D}_1$ is $\ll p^{1/2}$, and the lemma follows.
\end{proof}
\noindent
We combine the above results to get
\begin{lemma}\label{dc}
We have
\begin{equation}\label{7.13}
\mathcal{C}\ll
\begin{cases}
p^3\mathcal{C}_0, &\quad n_2=0, p|(n\bar{l}-n'\bar{l'})\,\, \hbox{and}\,\, p\nmid n_1\\
\frac{p^2qr}{n_1}\mathcal{A}, &\quad n_2=0, p\nmid (n\bar{l}-n'\bar{l'})\,\,\hbox{and} \,\,  p\nmid n_1\\
p^{3}|\mathcal{C}_2\mathcal{C}_3|, &\quad n_2=0(\bmod p), n_2\neq 0\,\,\hbox{and}\,\, p\nmid n_1
\end{cases}
\end{equation}
where
\begin{equation*}
\mathcal{C}_0=\sum_{d|q}\sum_{d'|q}dd'\sum_{\substack{\alpha (\frac{qr}{n_1})\\ n_1\alpha\equiv -n\bar{l}(d)\\n_1\alpha\equiv -n'\bar{l'}(d')}} 1\,,
\end{equation*}
\begin{equation*}
\mathcal{A}=\mathop{\sum\sum}_{\substack{d,d'|q\\ (d,d')|(nl'-n'l)}}(d,d')\,,
\end{equation*} and $\mathcal{C}_2$ and $\mathcal{C}_3$ as in \eqref{eq33} and \eqref{eq34} respectively.
\end{lemma}Note that the $p|n_1$ case is already being considered in \eqref{620}, which was independent of whether $n_2=0\bmod p$ or not.

\begin{proof}
First note that, from the congruence condition
\begin{equation}
q_2'\bar{\alpha}-q_2\bar{\alpha'}\equiv n_2 (\frac{rq_2q_2'q_1}{n_1})
\end{equation}in $\mathcal{D}_2$, $n_2=0$ implies $q_2=q_2'$ and $\alpha=\alpha'$. 
\\
\\
\noindent
Hence, in the first of case of \eqref{7.13}, $\mathcal{D}_0$ is the sum $\tilde{\mathcal{C}}$ in \eqref{71} with $c_1=c_2=n\overline{q^2l}, \beta=1$, multiplied by a factor of $\chi(\bar{l}l')$. In this case, the Lemma \ref{charb} gives
\begin{equation}\label{eq7.16}
\mathcal{D}_0\ll p^{2}.
\end{equation}Also we have
\begin{equation}\label{eq7.17}
\mathcal{D}_2\ll\sum_{d|q}\sum_{d'|q}dd'\sum_{\substack{\alpha (\frac{qr}{n_1})\\ n_1\alpha\equiv -n\bar{l}(d)\\n_1\alpha\equiv -n'\bar{l'}(d')}} 1\,,
\end{equation}Hence the first case of the Lemma follows from \eqref{7.9}, \eqref{7.12}, \eqref{eq7.16} and \eqref{eq7.17}.
\\
\\
\noindent
In the second case, $\mathcal{D}_0$ is the sum $\tilde{\mathcal{C}}$ in \eqref{71} with $c_1=n\overline{q^2l}, c_2=n'\overline{q^2l'}$ and $\beta=1$, multiplied by a factor of $\chi(\bar{l}l')$. In this case we have $\beta c_1\neq c_2\bmod p$ and hence Lemma \ref{charb} gives
\begin{equation}\label{eq48}
\mathcal{D}_0\ll p\,.
\end{equation}Also, we have
\begin{equation}\label{eq718}
\mathcal{D}_2\ll \sum_{d|q}\sum_{d'|q}dd'\sum_{\substack{\alpha (\frac{qr}{n_1})\\ n_1\alpha\equiv -n\bar{l}(d)\\n_1\alpha\equiv -n'\bar{l'}(d')}}1 \ll \frac{qr}{n_1}\mathop{\sum\sum}_{\substack{d,d'|q\\ (d,d')|(nl'-n'l)}}(d,d')\,
\end{equation} and the second part of the Lemma follows after combining \eqref{7.9}, \eqref{7.12}, \eqref{eq48} and \eqref{eq718}.
\\
\\
\noindent
In the last case, the number of such $n_2$ is $\ll N_2/p$. The bound in the Lemma in this case follows after using
\begin{equation}
\mathcal{D}_0\ll p^2
\end{equation}from Lemma \ref{charb} and observing that
\begin{equation}
\mathcal{D}_2\ll |\mathcal{C}_2\mathcal{C}_3|
\end{equation} since $(q_2q_2',rq_1/n_1)=1$, where $\mathcal{C}_2$ and $\mathcal{C}_3$ as in \eqref{eq33} and \eqref{eq34} respectively.
\end{proof}

We now estimate the contribution of each of the cases of Lemma \ref{dc} to \eqref{eq14}. Let $\sum_{0}, \sum_{\mathcal{A}}$ and $\sum_{\mathcal{B}}$ denote the corresponding contribution of the cases in the above lemma towards \eqref{eq14}.
\begin{lemma}\label{d}
We have
\begin{equation}
\hbox{$\sum_{0}\ll N^{1/2}pL^{1/2}r^{1/2}+ N^{1/2}pLr^{1/2}+N^{3/4}p^{3/4}r^{1/2}/L^{1/4}$.}
\end{equation}
\end{lemma}	
\begin{proof}
Recall that $n_2=0$ implies $q_2=q'_2$, and hence $q'=q$. Substituting the first bound from  Lemma \ref{dc} and rearranging, we see that the contribution of this case towards $\Omega$ in \eqref{eq21} is dominated by
\begin{equation}\label{nz}
\begin{aligned}
\frac{p^3M_0L^{4/3}}{n_1^2N_1^{1/2}}\mathop{\sum\sum}_{l,l'\sim L}|\overline{\lambda_{\pi}(1,l)}\lambda_{\pi}(1,l')|\sum_{\substack{q_2\sim C/q_1\\(q,pl)=1\\(q,pl')=1}}\sum_{d|q}\sum_{d'|q}dd'&\sum_{n\sim N_1}\sum_{\substack{n'\sim N_1\\p(d,d')|(nl'-n'l)}} |\lambda_f(n)\lambda_f(n')|\\
&\quad\quad\times\sum_{\substack{\alpha (\frac{qr}{n_1})\\ n_1\alpha\equiv -n\bar{l}(d)\\n_1\alpha\equiv -n'\bar{l'}(d')}}|\mathcal{J}|.
\end{aligned}
\end{equation}
As earlier in \eqref{amgm}, we use the inequality
\begin{equation}\label{Amgm}
|\lambda(n)\overline{\lambda(n')}|\ll |\lambda(n)|^2+|\lambda(n')|^2
\end{equation}and consider the contribution of the first term only.
From the last congruence condition in \eqref{nz}, $\alpha$ is determined modulo the l.c.m of $d/(n_1,d)$ and $d'/(n_1,d')$, and note that 
\begin{equation}
\frac{[d,d']}{[(n_1,d),(n_1,d')]}\leq\left[\frac{d}{(n_1,d)},\frac{d'}{(n_1,d')}\right]\leq \frac{q}{(n_1,q)}\,,
\end{equation}where $[a,b]$ denotes the l.c.m. of $a$ and $b$. Since $d,d'|q$ and $(n_1,q_2)=1$, we get $[(n_1,d),(n_1,d')]\leq (n_1,q_1)$, so that the first term in the above inequality is greater than $[d,d']/(n_1,q_1)$. Hence we get
\begin{equation}\label{alphacount}
\sum_{\substack{\alpha (\frac{qr}{n_1})\\ n_1\alpha\equiv -n\bar{l}(d)\\n_1\alpha\equiv -n'\bar{l'}(d')}}1\ll \frac{(n_1,q_1)}{[d,d']}.\frac{qr}{n_1}\ll \frac{qr}{[d,d']}\,.
\end{equation}We substitute this and the bound $(q/Q)^3$ for $\mathcal{J}$ in the last sum of \eqref{nz}.

We next execute the $n'$ sum according to the cases $p(d,d')\ll N_1$, $N_1\ll p(d,d')\ll N_1L$ and $p(d,d)\gg N_1L$. In the first case, $p(d,d')\ll N_1$ , we count $n'$ satisfying the second last congruence in \eqref{nz} to get
\begin{equation}
\begin{aligned}
\frac{p^3M_0L^{4/3}}{n_1^2N_1^{1/2}}\times \frac{q^3}{Q^3}\times qr\times&\mathop{\sum\sum}_{l,l'\sim L}|\overline{\lambda_{\pi}(1,l)}\lambda_{\pi}(1,l')|\\
&\times\sum_{\substack{q_2\sim C/q_1\\(q,pl)=1\\(q,pl')=1}}\sum_{d|q}\sum_{d'|q}\frac{dd'}{[d,d']}\sum_{n\sim N_1}\frac{N_1}{p(d,d')}|\lambda_f(n)|^2\,.
\end{aligned}
\end{equation}
For the $n$ sum, we can now use the Ramanujan bound on average and execute the remaining sum to get
\begin{equation}\label{7.17}
\begin{aligned}
&\ll \frac{M_0L^{4/3}}{n_1^2N_1^{1/2}}\times\frac{q^3}{Q^3}\times p^2qr\times\sum_{\substack{l,l'\sim L}}|\overline{\lambda_{\pi}(1,l)}\lambda_{\pi}(1,l')|\sum_{q_2\sim C/q_1}\sum_{d|q}\sum_{d'|q'}N_1^2\,,\\
&\ll \frac{M_0L^{4/3}}{n_1^2N_1^{1/2}}\times\frac{q^3}{Q^3}\times p^2qr\times \frac{L^2CN_1^2}{q_1}\,
\end{aligned}
\end{equation}where we have used the Cauchy-Schwarz inequality and \eqref{rama} for the $l,l'$ sum. Substituting $N_1\ll pL$ and $Q=\sqrt{(NL)/p}$, we see that the above is bounded by
\begin{equation}
\frac{p^{5}M_0L^{2+4/3}rC^5}{n_1^2q_1N^{3/2}}\,.
\end{equation}

Substituting, we see that the contribution of this case in \eqref{eq14} is
\begin{equation}\label{eq24}
\begin{aligned}
&\ll \frac{N^{17/12}}{r^{2/3}p^3QC^{5/2}L}\left(\frac{p^{5}L^{2+4/3}rC^5}{N^{3/2}}\right)^{1/2}M_0^{1/2}\sum_{n_1\ll Cpr}\frac{n_1^{1/3}\Theta^{1/2}}{n_1}\sum_{\frac{n_1}{(n_1,pr)}|q_1|(pn_1)^{\infty}}\frac{1}{q_1^{1/2}}\,,\\
&\ll \frac{N^{1/6}L^{1/6}}{r^{1/6}} M_0^{1/2+1/6}\,,
\end{aligned}
\end{equation}where we used \eqref{632} in the last line. Substituting $M_0=rp^{3/2+\epsilon}N^{1/2}L^{1/2}$ gives the first term of the lemma. 

In the second case ,$N_1\ll p(d,d')\ll N_1L$, for each $n$, there is at most one $n'\sim N_1$ satisfying the congruence relation $p(d,d')|(nl'-n'l)$ (depends on $l,l'$ as well). So, \eqref{nz} in this case (recall that we are considering the contribution only the first term in the right hand side of \eqref{Amgm}) is bounded by
\begin{equation}\label{nz1}
\begin{aligned}
\frac{p^3M_0L^{4/3}}{n_1^2N_1^{1/2}}\times \frac{q^3}{Q^3}\times qr\mathop{\sum\sum}_{l,l'\sim L}|\overline{\lambda_{\pi}(1,l)}\lambda_{\pi}(1,l')|&\sum_{\substack{q_2\sim C/q_1\\(q,pl)=1\\(q,pl')=1}}\sum_{d|q}\sum_{\substack{d'|q\\N_1\ll p(d,d')\ll N_0L }}\frac{dd'}{[d,d']}\\
&\times\sum_{n\sim N_1}|\lambda_f(n)|^2\,.
\end{aligned}
\end{equation} Using the Ramanujan bound on average on the $n$ sum, we bound the above sum by
\begin{equation}\label{721}
\frac{p^2M_0L^{4/3}}{n_1^2N_1^{1/2}}\times\frac{q^3}{Q^3}\times pqr\mathop{\sum\sum}_{l,l'\sim L}|\overline{\lambda_{\pi}(1,l)}\lambda_{\pi}(1,l')|\sum_{q_2\sim C/q_1}\sum_{d|q}\sum_{\substack{d'|q\\N_1\ll p(d,d')\ll N_0L }}(d,d')N_1\,.
\end{equation}
We use the upper bound $(d,d')\ll N_0L/p$ and execute the remaining sum trivially. It is then clear that when $N_1$ is replaced by $N_0$, \eqref{721} differs from \eqref{7.17} by only a factor of $L$. Hence multiplying its square root, $L^{1/2}$,  with the first term of the lemma, gives the second term of the lemma,
\begin{equation}\label{eq26}
 N^{1/2}pLr^{1/2}.
\end{equation}
For the last case $p(d,d)\gg N_1L$,  we must have $nl'-n'l=0$ . In this case we keep the term $|\lambda_f(n)\lambda_f(n')|$ as it is. Since $l$ and $l'$ are primes, there are two possible cases, either $n=n', l=l'$ or $n=\tilde{n}l, n'=\tilde{n}l'$ for some $\tilde{n}\ll N_1/L$. In the latter case, \eqref{nz} is bounded by
\begin{equation}\label{zf}
\begin{aligned}
\ll \frac{p^3M_0L^{4/3}}{n_1^2N_1^{1/2}}\times\frac{q^3}{Q^3}\times qr&\mathop{\sum\sum}_{l,l'\sim L}|\overline{\lambda_{\pi}(1,l)}\lambda_{\pi}(1,l')|\\
&\times\sum_{q_2\sim C/q_1}\sum_{d|q}\sum_{\substack{d'|q\\N_1L\ll p(d,d')}}(d,d')\sum_{\tilde{n}\sim N_1/L}|\lambda_f(\tilde{n}l)\lambda_f(\tilde{n}l')|\,.
\end{aligned}
\end{equation} Using the Hecke relation, one has
\begin{equation}
\lambda_f(\tilde{n}l)=\lambda_f(\tilde{n})\lambda_f(l)-\lambda_f(\tilde{n}/l)\,,
\end{equation}where the second term exists only if $l|\tilde{n}$.
Hence
\begin{equation*}
\sum_{\tilde{n}\sim N_1/L}|\lambda_f(\tilde{n}l)|^2\ll |\lambda_f(l)|^2\sum_{\tilde{n}\sim N_1/L}|\lambda_f(\tilde{n})|^2+\sum_{n'\sim N_1/L^2}|\lambda_f(n')|^2\ll \frac{|\lambda_f(l)|^2N_1}{L}+\frac{N_1}{L^2}\,.
\end{equation*} Using these bounds after an application of Cauchy-Schwarz inequality  in \eqref{zf}, we get
\begin{equation}\label{725}
\begin{aligned}
\frac{p^3M_0L^{4/3}}{n_1^2N_1^{1/2}}\times\frac{q^3}{Q^3} \times qr&\mathop{\sum\sum}_{l,l'\sim L}|\overline{\lambda_{\pi}(1,l)}\lambda_{\pi}(1,l')|\sum_{q_2\sim C/q_1}\sum_{d|q}\sum_{\substack{d'|q\\N_1L\ll p(d,d')}}(d,d')\\
&\times\left(\frac{|\lambda_f(l)|N_1^{1/2}}{L^{1/2}}+\frac{N_1^{1/2}}{L}\right)\left(\frac{|\lambda_f(l')|N_1^{1/2}}{L^{1/2}}+\frac{N_1^{1/2}}{L}\right).
\end{aligned}
\end{equation}Executing the remaining sum, we see that the above is dominated by
\begin{equation}\label{eq735}
 \frac{M_0L^{4/3}}{n_1^2N_1^{1/2}}\times\frac{q^3}{Q^3}\times p^3qr\times \frac{C^2}{q_1}\times N_1L\,,
\end{equation}where we have used the Cauchy-Schwarz inequality and \eqref{rama} for the $l,l'$ sum. \\
For the case with $n=n', l=l'$, using the bound from \eqref{alphacount}, \eqref{nz} reduces to
\begin{equation}\notag
\frac{p^3M_0L^{4/3}}{n_1^2N_1^{1/2}}\times\frac{q^3}{Q^3}\times qr\sum_{l\sim L}|\overline{\lambda_{\pi}(1,l)}|^2\sum_{q_2\sim C/q_1}\sum_{d|q}\sum_{\substack{d'|q\\N_1L\ll p(d,d')}}(d,d')\sum_{n\sim N_1}|\lambda_f(n)|^2.
\end{equation}Using the Ramanujan bound on average for $l$ and $n$ sum, and executing the remaining sum trivially, we see that this case produces the same bound as \eqref{eq735}.

Note that when $N_1$ is replaced with $N_0=p^{1+\epsilon}L$, and $C$ with $Q=(NL/p)^{1/2}$, \eqref{eq735} differs from \eqref{7.17} by a factor of $(Qp)/(N_0L)\sim N^{1/2}/(p^{1/2}L^{3/2})$. Hence multiplying its square root with the first term of the lemma, we see that this case contributes 
\begin{equation}\label{eq28}
N^{3/4}p^{3/4}r^{1/2}/L^{1/4}\,
\end{equation} towards \eqref{eq14}, which is the third term of the lemma. This completes the proof the lemma.
\end{proof} 
\begin{lemma}\label{ndd}
We have
\begin{equation*}
\hbox{$\sum_{\mathcal{A}} \,,\,\,\sum_{\mathcal{B}}\ll N^{3/4}p^{1/2}L^{3/4}r^{1/2}$}\,.
\end{equation*}
\end{lemma}	
\begin{proof}
For $\sum_{\mathcal{A}}$, substituting the second bound from Lemma \ref{dc}, we see that the contribution of this case to $\Omega$ is bounded by
\begin{equation}\label{sigmaA}
\begin{aligned}
&\frac{p^2qrM_0L^{4/3}}{n_1^3N_1^{1/2}}\frac{q^3}{Q^3}\mathop{\sum}_{l,l'\sim L}|\overline{\lambda_{\pi}(1,l)}\lambda_{\pi}(1,l')|\sum_{\substack{q_2\sim C/q_1\\(q,pl)=1\\(q,pl')=1}}\sum_{d|q}\sum_{d'|q}(d,d')\sum_{n\sim N_1}\sum_{\substack{n'\sim N_1\\(d,d')|(nl'-n'l)}} |\lambda_f(n)\lambda_f(n')|.
\end{aligned}
\end{equation}We again appeal to the inequality
\begin{equation}\notag
|\lambda(n)\overline{\lambda(n')}|\ll |\lambda(n)|^2+|\lambda(n')|^2
\end{equation}and work out the details only for the first term. We obtain that \eqref{sigmaA} is bounded by
\begin{equation}\notag
\frac{p^{2}qrM_0L^{4/3}}{n_1^3N_1^{1/2}}\frac{q^3}{Q^3}\mathop{\sum\sum}_{l,l'\sim L}|\overline{\lambda_{\pi}(1,l)}\lambda_{\pi}(1,l')|\sum_{\substack{q_2\sim C/q_1\\(q,pl)=1\\(q,pl')=1}}\sum_{d|q}\sum_{d'|q}(d,d')\sum_{n\sim N_1}\sum_{\substack{n'\sim N_1\\(d,d')|(nl'-n'l)}}|\lambda(n)|^2\,.
\end{equation}Counting the number of $n'$ with given constraint and executing the $n$ sum using the Ramanujan bound on average, we get
\begin{equation*}
\begin{aligned}
&\ll \frac{p^{2}qrM_0L^{4/3}}{n_1^3N_1^{1/2}}\times \frac{q^3}{Q^3}\mathop{\sum\sum}_{l,l'\sim L}|\overline{\lambda_{\pi}(1,l)}\lambda_{\pi}(1,l')|\sum_{q_2\sim C/q_1}\sum_{d|q}\sum_{d'|q}(d,d')\left(N_1+\frac{N_1^2}{(d,d')}\right)\,,\\
&\ll \frac{p^{2}CrM_0L^{2+4/3}}{n_1^3N_1^{1/2}}\times \frac{q^3}{Q^3}\left(\frac{C^2N_1}{q_1}+\frac{CN_1^2}{q_1}\right),
\end{aligned}
\end{equation*}which is smaller than the first bound in \eqref{sl} by a factor of $p^{1/2}N_2q_1$. Consequently, $\sum_{\mathcal{A}}$ gets absorbed by the bound in the second line of \eqref{sl}.\\
\\
\noindent
Finally, for $\sum_{\mathcal{B}}$, recall that $p|n_2, n_2\neq 0$. The number of such $n_2$ is $\ll N_2/p$. Using the earlier obtained bounds for $\mathcal{C}_2, \mathcal{C}_3$ from Lemma \ref{c2} and Lemma \ref{c3}, and doing the same calculation as in \eqref{eq38}-\eqref{6.29}, with the only changes $p^{3}$ instead of $p^{5/2}$, and $N_2/p$ instead of $N_2$, we arrive at a bound for $\sum_{\mathcal{B}}$ which is smaller than \eqref{sl} by factor of $p^{1/2}$.
\end{proof}

\section{The special cases }
In this section we discuss the cases complementary to \eqref{210}. Treatment of these cases is similar to the general case above and yields smaller contributions. 
\subsection{$u=0$ and\,$(pl,q)=1$} In this case, the conductor of the GL(3) variable drops and we get huge saving from the Voronoi summation formulas itself. Hence, trivially estimating the dual sum which arises on applying the Voronoi formulae is enough for this case.
The $m$ sum in \eqref{eq5} becomes
\begin{equation}
\sum_{m}\lambda_{\pi}(r,m)e\left(\frac{ma}{q}\right)e\left(\frac{mx}{pqQ}\right)V\left(\frac{m}{lN}\right)
\end{equation}Hence, after the application of the $GL(3)$ Voronoi formula, the $m$ sum \eqref{msum} changes to
\begin{equation}
\begin{aligned}
\frac{(Nl)^{2/3}}{qr^{2/3}}\sum_{\pm}\sum_{n_1|qr}n_1^{1/3}\sum_{n_2=1}^{\infty}\frac{\lambda_{\pi}(n_1,n_2)}{n_2^{1/3}}&S\left(r\overline{a},\pm n_2;qr/n_1\right)\\ &\times
\int_{\mathbb{R}}V(z)e\left(\frac{Nlxz}{pqQ}\pm \frac{3(Nln_1^2n_2z)^{1/3}}{qr^{1/3}}\right)dz
\end{aligned}
\end{equation}where now the dual variables restrict up to $n_1^2n_2\ll M_0$, where $M_0=(rN^2L^2)/(p^3Q^3)=(rN^{1/2}L^{1/2})/p^{3/2}$. The treatment for the $GL(2)$ sum remains the same (with $u=0$). So the contribution of this case to $S_r(N)$ sum in \eqref{comb} is
\begin{equation}\label{83}
\begin{aligned}
\frac{N^{\frac{3}{4}+\frac{2}{3}}}{g_{\bar{\chi}}p^{\frac{3}{2}}r^{\frac{2}{3}}QL}\sum_{l\in\mathscr{L}}\overline{\lambda_{\pi}(1,l)}l^{\frac{2}{3}}\sum_{1\leq q\leq Q}\frac{1}{q^{5/2}}&\sum_{n_1|qr}n_1^{1/3}\sum_{n_2\ll\frac{M_0}{n_1^2}}\frac{\lambda_{\pi}(n_1,n_2)}{n_2^{1/3}}\\
&\sum_{n\ll N_0}\frac{\lambda(n)}{n^{1/4}}C_2(n\bar{l},n_1,n_2,q)\,\,I(n_1^2n_2,n,q)\,
\end{aligned}
\end{equation}where $N_0=pL$ as earlier, and
\begin{equation}
I(m,n,q)=\int\int\int g(q,x)V(z)U(y)e\left(\frac{Nlx(z-y)}{pqQ}+\frac{2\sqrt{nNy}}{pq}+\frac{3(Nlmz)^{1/3}}{qr^{1/3}}\right)dy\,dz\,dx\,,
\end{equation}and
\begin{equation}
C_2(n,n_1,n_2,q)=\sideset{}{^*}\sum_{a (q)}S(r\bar{a},n_2,qr/n_1)C_1(n,a,q)\,,
\end{equation}where
\begin{equation}
C_1(n,a,q)=\sum_{b(p)}\bar{\chi}(b)e\left(\frac{n(\overline{ap-b\bar{l}q})}{pq}\right)=e\left(\frac{n\bar{a}\bar{p}^2}{q}\right)\bar{\chi}(-n\bar{q}^2l)g_{\chi}\,,
\end{equation}so that
\begin{equation}
C_2(n,n_1,n_2,q)=\bar{\chi}(-n\bar{q}^2l)g_{\chi}\sideset{}{^*}\sum_{\alpha (qr/n_1)}e\left(\frac{\bar{\alpha}n_2n_1}{qr}\right)\left(\sum_{\substack{d|q\\n_1\alpha=-n\bar{p}^2\bmod d}}d\mu(q/d)\right)\,.
\end{equation}Hence we get
\begin{equation}
C_2(n\bar{l},n_1,n_2,q)\ll p^{1/2}\sum_{d|q}d\sideset{}{^*}\sum_{\substack{\alpha (qr/n_1)\\n_1\alpha=-n\bar{l}\bar{p}^2\bmod d}}1\ll \frac{p^{1/2}qr}{n_1}\sum_{d|q}(d,n)\ll\frac{p^{1/2}qr(q,n)}{n_1}\,.
\end{equation}Following the arguments used in the proof of Lemma \ref{int}, \eqref{u} and the comment before \eqref{5.24} in particular, we get
\begin{equation}
I(m,n,q)\ll pqQ/NL\,.
\end{equation}Substituting these two bounds, we see that the second line of \eqref{83} is bounded by 
\begin{equation}
\frac{p^{3/2}Qq^2r}{NLn_1}\sum_{n\ll N_0}\frac{|\lambda(n)|(q,n)}{n^{1/4}}.
\end{equation}Consequently, \eqref{83} is bounded by
\begin{equation}
\frac{N^{17/12}L^{2/3}r}{ NL^2p^{1/2}r^{2/3}}\sum_{l\in\mathscr{L}}|\overline{\lambda_{\pi}(1,l)}|\sum_{n_1\ll Qr}\frac{1}{n_1^{2/3}}\sum_{n_2\ll M_0/n_1^2}\frac{|\lambda_{\pi}(n_1,n_2)|}{n_2^{1/3}}\sum_{n\ll N_0}\frac{|\lambda(n)|}{n^{1/4}}\sum_{\substack{q\leq Q\\ \frac{n_1}{(n_1,r)}|q}}\frac{(q,n)}{q^{1/2}}.
\end{equation}The last $q$-sum is bounded by $ Q^{1/2}$. By partial summation and the Ramanujan bound on average \eqref{rama}, the $n$-sum is then bounded by $N_0^{3/4}$. Substituting these two bounds, we get
\begin{equation}\label{812}
\frac{N^{17/12}L^{2/3}N_0^{3/4}Q^{1/2}r}{ NL^2p^{1/2}r^{2/3}}\sum_{l\in\mathscr{L}}|\overline{\lambda_{\pi}(1,l)}|\sum_{n_1^2n_2\ll M_0}\frac{|\lambda_{\pi}(n_1,n_2)|}{(n_1^2n_2)^{1/3}}
\end{equation}Using the Ramanujan bound on average \eqref{rama} and partial summation, we see that the last sum in \eqref{812} is bounded by $M_0^{2/3}$. Also, the sum over $\mathscr{L}$ is bounded by $L$ by \eqref{rama}. Substituting $M_0=p^{-3/2}N^{1/2}L^{1/2}r, N_0=pL$ and $Q=\sqrt{(NL)/p}$, we see that the contribution of this case towards $S_r(N)$ in \eqref{comb} is bounded by 
\begin{equation}
\frac{rNL}{p}\ll N^{3/4}p^{1/2}L^{3/4}r^{1/2}
\end{equation}where the last inequality holds since $N\ll p^3/r^2$ and $L<p$ (see \eqref{par}, \eqref{eq2.8}). Recall that the term on the right hand side is the earlier bound from \eqref{nd}.

\subsection{$(pl,q)>1$}In this subsection, we will focus only on the character sum changes and pass over the various technical points that have been already presented in full detail in the above sections.

Note that since $p,l$ are primes and $q\leq Q=(NL/p)^{1/2}\leq pL^{1/2}$, either $l|q$ or $p|q$. In the former case, we replace $q$ with $q'=q/l\sim C/l$ and proceed the same as the general case above to obtain a smaller contribution. Except for a minor change in the $GL(2)$ Voronoi summation, the treatment for everything else is pretty much the same, and we omit the details for the sake of simplicity. We proceed with the details for the latter case.
\subsubsection{$p|q$ and $(l,q)=1$}Writing $q= pq'$, where 
\begin{equation}
q'\leq Q/p \leq L^{1/2}< p^{1/2},
\end{equation}
\eqref{eq5} is now of the form
\begin{equation}\label{815}
\begin{aligned}
\frac{1}{p^2QL}\sum_{u\bmod p}\sum_{q'\leq Q/p }\frac{1}{q'}\,\,\sideset{}{^*}\sum_{a\bmod pq'}\sum_{l\sim L}\overline{\lambda_{\pi}(1,l)}&\sum_{m\sim NL}\lambda_{\pi}(r,m)e\left(\frac{m(a+upq')}{p^2q'}\right)\\
&\sum_{n\sim N}\lambda(n)\chi(n)e\left(\frac{-nl(a+upq')}{p^2q'}\right).
\end{aligned}
\end{equation} Replacing $a+upq'\mapsto a \,\,(\bmod p^2q') $, \eqref{815} becomes
\begin{equation}\label{816}
\begin{aligned}
\frac{1}{p^2QL}\sum_{q'\leq Q/p }\frac{1}{q'}\,\,\sideset{}{^*}\sum_{a\bmod\,\,p^2q'}\sum_{l\sim L}\overline{\lambda_{\pi}(1,l)}.\overline{\chi(l)}&\sum_{m\sim NL}\lambda_{\pi}(r,m)e\left(\frac{ma}{p^2q'}\right)\\
&\sum_{n\sim N}\lambda(n)\chi(nl)e\left(\frac{-nla}{p^2q'}\right).
\end{aligned}
\end{equation}The $GL(3)$ and $GL(2)$ Voronoi formulae transforms the $m$ and $n$ sum essentially into
\begin{equation}
\frac{NL}{(p^2q')^2r^{2/3}}\sum_{\tilde{m}\ll (p^2q')^3/(NL)}\lambda_{\pi}(1,\tilde{m})S(r\bar{a},\pm \tilde{m}; p^2q'r)
\end{equation}and
\begin{equation}
\frac{N}{p^2q'g_{\bar{\chi}}}\sideset{}{^*}\sum_{b (p)}\bar{\chi}(b)\sum_{\tilde{n}\ll (p^2q')^2/N}\lambda(\tilde{n})e\left(\frac{-\tilde{n}\overline{l(a-bpq')}}{p^2q'}\right)
\end{equation}respectively. Hence, \eqref{816} essentially becomes
\begin{equation}\label{819}
\frac{N^2}{p^8Qg_{\bar{\chi }}r^{2/3}}\sum_{q'\leq Q/p }\frac{1}{q'^4}\,\,\sum_{l\sim L}\overline{\lambda_{\pi}(1,l)}.\overline{\chi(l)}\sum_{\tilde{m}\sim (p^2q')^3/(NL) }\,\,\sum_{\tilde{n}\ll (p^2q')^2/N}\lambda_{\pi}(1,\tilde{m})\lambda(\tilde{n})\mathfrak{C}\,\,,
\end{equation}where the character sum $\mathfrak{C}$ is given by
\begin{equation}\label{chardef}
\mathfrak{C}= \sideset{}{^*}\sum_{b (p)}\sideset{}{^*}\sum_{a (p^2q')}\bar{\chi}(b)e\left(\frac{-\tilde{n}\overline{l(a-bpq')}}{p^2q'}\right)S(r\bar{a},\pm \tilde{m}; p^2q'r).
\end{equation}
We have the following reduction of $\mathfrak{C}$ into additive character with respect to the $\tilde{m}$ variable :
\begin{lemma}\label{•}
We have
\begin{equation}\label{821}
\mathfrak{C}=p^2g_{\bar{\chi}}\chi(-\overline{\tilde{n}^3r^2}l^3\tilde{m}^2)\sideset{}{^*}\sum_{x_1 (q'r)}f(x_1, \tilde{n}\bar{l}; q')e\left(\frac{\pm\tilde{m}(p^2\bar{p}^2\bar{x_1}+q'r\overline{q'r}\overline{\tilde{n}}l)}{p^2q'r}\right),
\end{equation}where
\begin{equation}\label{fdef}
f(x,n,q)= \sideset{}{^*}\sum_{a (q)}e\left(\frac{a(x-n)}{q}\right)=\sum_{\substack{d|q\\ x= n (\bmod d)}}d\mu(q/d).
\end{equation}
\end{lemma}
\begin{proof}
Firstly, note that
\begin{equation}\notag
\overline{a-bpq'}=\bar{a}+\bar{a}^2bpq'\,(\bmod p^2q').
\end{equation}Substituting and executing the $b$-sum in \eqref{chardef}, which is a gauss sum, we obtain
\begin{equation}\notag
\begin{aligned}
\mathfrak{C}&=\sideset{}{^*}\sum_{b (p)}\sideset{}{^*}\sum_{a (p^2q')}\bar{\chi}(b)e\left(-\frac{\tilde{n}\overline{la}}{p^2q'}-\frac{\tilde{n}b\overline{la^2}}{p}\right)S(r\bar{a},\pm \tilde{m}; p^2q'r)\\
&=g_{\bar{\chi}}\chi(-\tilde{n}\overline{l})\sideset{}{^*}\sum_{a (p^2q')}\chi(\bar{a}^2)e\left(-\frac{\tilde{n}\overline{la}}{p^2q'}\right)S(r\bar{a},\pm \tilde{m}; p^2q'r).
\end{aligned}
\end{equation}Opening up the Kloosterman sum we obtain
\begin{equation}\label{CC}
\begin{aligned}
\mathfrak{C}&=g_{\bar{\chi}}\chi(-\tilde{n}\overline{l})\sideset{}{^*}\sum_{a (p^2q')}\,\,\sideset{}{^*}\sum_{x (p^2q'r)}\chi(\bar{a}^2)e\left(-\frac{\tilde{n}\overline{q'la}}{p^2}-\frac{\tilde{n}\overline{p^2la}}{q'}+\frac{x\overline{q'a}}{p^2}+\frac{x\overline{p^2a}}{q'}\pm\frac{\bar{x}\tilde{m}\overline{q'r}}{p^2}\pm\frac{\bar{x}\tilde{m}\overline{p^2}}{q'r}\right)\\
&=g_{\bar{\chi}}\chi(-\tilde{n}\overline{l})\cdot\mathcal{C}_2\sideset{}{^*}\sum_{x_1 (q'r)}\mathcal{C}_1(x_1)e\left(\pm\frac{\bar{x_1}\tilde{m}\overline{p^2}}{q'r}\right),
\end{aligned}
\end{equation}where
\begin{equation}\label{CC1}
\mathcal{C}_1(x_1):= \sideset{}{^*}\sum_{a_1(q')}e\left(\frac{-\tilde{n}\overline{p^2la_1}}{q'}+\frac{x_1\overline{p^2a_1}}{q'}\right),
\end{equation}and
\begin{equation}\label{C2}
\mathcal{C}_2:=\sideset{}{^*}\sum_{x_2(p^2)}e\left(\pm\frac{\bar{x_2}\tilde{m}\overline{q'r}}{p^2}\right)\sideset{}{^*}\sum_{a_2(p^2)}\chi(\bar{a_2}^2)e\left(-\frac{\tilde{n}\overline{q'la_2}}{p^2}+\frac{x_2\overline{q'a_2}}{p^2}\right).
\end{equation}Note that by definition, $\mathcal{C}_1(x_1)=f(x_1, \tilde{n}\bar{l}; q')$, where $f$ defined as in \eqref{fdef}. To evaluate $\mathcal{C}_2$, we write $\bar{a_2}=p\alpha+\beta\bmod p^2$, where $\alpha=0,1,\cdots, p-1$ and $\beta=1,2,\cdots, p-1$, to see that the $a_2$ sum in \eqref{C2} becomes
\begin{equation}\notag
\begin{aligned}
&\sideset{}{^*}\sum_{\beta (p)}\chi(\beta^2)e\left(\frac{(-\tilde{n}\overline{q'l}+x_2\overline{q'})\beta}{p^2}\right)\sum_{\alpha (p)}e\left(\frac{(-\tilde{n}\overline{q'l}+x_2\overline{q'})\alpha}{p}\right)\\
&=p\delta_{x_2=\tilde{n}\bar{l}\bmod p}\sideset{}{^*}\sum_{\beta (p)}\chi(\beta^2)e\left(\frac{((-\tilde{n}\overline{q'l}+x_2\overline{q'})/p)\beta}{p}\right)\\
&=pg_{\chi^2}\cdot\delta_{x_2=\tilde{n}\bar{l}\bmod p}\cdot\bar{\chi}^2\left(\frac{-\tilde{n}\overline{q'l}+x_2\overline{q'}}{p}\right).
\end{aligned}
\end{equation}Writing $x_2=p\gamma+\tilde{n}\bar{l}, \gamma=0,1,\cdots,p-1$, and substituting the last expression into \eqref{C2}, we obtain
\begin{equation}\label{finalC2}
\begin{aligned}
\mathcal{C}_2&=pg_{\chi^2}\chi^2(q'^2)\sum_{\gamma (p)}\bar{\chi}^2(\gamma)e\left(\pm\frac{\tilde{m}\overline{q'r}\overline{(p\gamma+\tilde{n}\bar{l})}}{p^2}\right)\\
&=pg_{\chi^2}\chi^2(q'^2)\sum_{\gamma (p)}\bar{\chi}^2(\gamma)e\left(\pm\frac{\tilde{m}l\overline{q'r\tilde{n}}(1+p\overline{\tilde{n}}l\gamma)}{p^2}\right)\\
&=pg_{\chi^2}g_{\bar{\chi}^2}\chi^2(q'^2)\chi^2(\tilde{m}l^2\overline{q'r\tilde{n}^2})e\left(\pm\frac{\tilde{m}l\overline{q'r\tilde{n}}}{p^2}\right).
\end{aligned}
\end{equation}The lemma follows after substituting \eqref{finalC2} and \eqref{CC1} in \eqref{CC}.
\end{proof}
Like earlier, we now apply the Cauchy-Schwarz inequality to \eqref{819} to get rid of the $GL(3)$ coefficient, but this time we keep the $q'$-sum outside the square and $\tilde{n},l$ sum inside. We get that \eqref{819} is bounded by
\begin{equation}\label{8.23}
\frac{N^2}{p^8Qg_{\bar{\chi}}r^{2/3}}\sum_{q'\leq Q/p }\frac{1}{q'^4}\,\frac{(p^2q')^{3/2}}{(NL)^{1/2}}\,\,\Omega^{1/2},
\end{equation}where
\begin{equation}
\Omega = \sum_{\tilde{m}\sim (p^2q')^3/(NL) }\left|\sum_{l\sim L}\sum_{\tilde{n}\ll (p^2q')^2/N}\overline{\lambda_{\pi}(1,l)}.\overline{\chi(l)}\lambda(\tilde{n})\mathfrak{C}\right|^2
\end{equation}Opening the absolute value square and substituting \eqref{821} we arrive at
\begin{equation}\label{825}
\begin{aligned}
\Omega &=p^4|g_{\bar{\chi}}|^2\sum_{l,l'\sim L}\overline{\lambda_{\pi}(1,l)}\lambda_{\pi}(1,l')\chi(l^2\bar{l'}^2)\sum_{n,n'\ll (p^2q')^2/N }\lambda(n)\overline{\lambda(n')}\chi(\bar{n}^3n'^3)\\
&\qquad\qquad\qquad\times\sum_{x_1 (q'r)}\sum_{x_2 (q'r)}f(x_1,n\bar{l},q')\bar{f}(x_2,n'\bar{l'},q')\\
&\qquad\qquad\qquad\times\sum_{\tilde{m}\sim (p^2q')^3/(NL)}e\left(\frac{\tilde{m}(p^2\bar{p}^2(\bar{x_1}-\bar{x_2})+q'r\overline{q'r}(\overline{n}l-\overline{n'}l')}{p^2q'r}\right)\,.
\end{aligned}
\end{equation}We now apply Poisson summation in the $\tilde{m}$-sum in \eqref{825} with modulus $p^2q'r$. We observe that 
\begin{equation}\label{826}
\frac{(p^2q')^3}{NL}= p^2q'\frac{(p^2q')^2}{NL}\geq p^2q'\frac{p^4}{NL}\geq p^2q'r^2 \frac{p^{1-\epsilon}}{L}\geq p^{\epsilon}p^2q' r,
\end{equation}where we have used $N\leq p^{3+\epsilon}/r^2$ and $L=p^{\eta}, \eta<1$ for the last two inequality in \eqref{826}. Hence, it follows that only the zero-frequency contributes non-negligibly in dual side of the Poisson summation of the $\tilde{m}$ sum in \eqref{825}. In other words,
\begin{equation}
\begin{aligned}
\sum_{\tilde{m}\sim (p^2q')^3/(NL)}&e\left(\frac{\tilde{m}(p^2\bar{p}^2(\bar{x_1}-\bar{x_2})+q'r\overline{q'r}(\overline{n}l-\overline{n'}l')}{p^2q'r}\right)\\
&\ll \frac{(p^2q')^3}{NL}\delta_{p^2\bar{p}^2(\bar{x_1}-\bar{x_2})+q'r\overline{q'r}(\overline{n}l-\overline{n'}l')= 0\,\,(\bmod p^2q'r) }.
\end{aligned}
\end{equation}The above congruence condition implies 
\begin{equation}\label{828}
x_1=x_2,\,\,\,\,\,\text{and}\,\,\,\, nl'-nl=0\,(\bmod {p^2}).
\end{equation}Since $|nl'-nl|\ll(p^2q')^2/N\leq pL^2< p^2$ (see line after \eqref{94}), the second congruence in \eqref{828} implies
\begin{equation}
nl'=n'l.
\end{equation}Hence
\begin{equation}
\begin{aligned}
\Omega &\ll p^4|g_{\bar{\chi}}|^2\frac{(p^2q')^3}{NL}\sum_{l,l'\sim L}|\lambda_{\pi}(1,l)\lambda_{\pi}(1,l')|\sum_{\substack{n,n'\ll (p^2q')^2/N\\nl'=n'l}}|\lambda(n)\lambda(n')|\\
&\qquad\qquad\qquad\times\sum_{x_1 (q'r)}|f(x_1,n\bar{l},q')\bar{f}(x_1,n'\bar{l'},q')|\,.
\end{aligned}
\end{equation}But we have
\begin{equation}
\begin{aligned}
\sum_{x_1 (q'r)}|f(x_1,n\bar{l},q')\bar{f}(x_1,n'\bar{l'},q')|&\leq \sum_{d_1|q'r}\sum_{d_2|q'r}d_1d_2\sideset{}{^*}\sum_{\substack{x_1 (q'r)\\ x_1=n\bar{l} (d_1)\\x_1=n'\bar{l'} (d_2)}}1\\
&\ll \sum_{d_1|q'r}\sum_{d_2|q'r}d_1d_2 \frac{q'r}{[d_1,d_2]}\\
&\ll q'r\sum_{d_1|q'r}\sum_{d_2|q'r}(d_1,d_2)\\
&\ll (q'r)^2.
\end{aligned}
\end{equation}Substituting we get
\begin{equation}\label{8.32}
\begin{aligned}
\Omega &\ll p^4|g_{\bar{\chi}}|^2.\frac{(p^2q')^3}{NL}.(q'r)^2\sum_{l,l'\sim L}|\lambda_{\pi}(1,l)\lambda_{\pi}(1,l')|\sum_{\substack{n,n'\ll (p^2q')^2/N\\nl'=n'l}}|\lambda(n)\lambda(n')|.
\end{aligned}
\end{equation}From the earlier calculations \eqref{zf}-\eqref{725} we have
\begin{equation}
\sum_{\substack{n,n'\ll (p^2q')^2/N\\nl'=n'l}}|\lambda(n)\lambda(n')|\ll \left(\frac{|\lambda_f(l)|p^2q'}{L^{1/2}N^{1/2}}+\frac{p^2q'}{N^{1/2}L}\right)\left(\frac{|\lambda_f(l')|p^2q'}{L^{1/2}N^{1/2}}+\frac{p^2q'}{N^{1/2}L}\right).
\end{equation}Substituting and executing the remaining $l,l'$ sum in \eqref{8.32} using the Cauchy-Schwarz inequality and the Ramanujan bound on average \eqref{rama} we get
\begin{equation}
\Omega\ll p^5.\frac{(p^2q')^3}{NL}. (q'r)^2.\frac{(p^2q')^2}{N}.L=\frac{p^{15}q'^{7}r^2}{N^2}.
\end{equation}Substituting in \eqref{8.23}, we see that the contribution of this case to $S_r(N)$ is bounded by
\begin{equation}
\begin{aligned}
&\quad\frac{N^2}{p^8Q|g_{\bar{\chi}}|r^{2/3}}\sum_{q'\leq Q/p }\frac{1}{q'^4}\,\frac{(p^2q')^{3/2}}{(NL)^{1/2}}\left(\frac{p^{15}q'^{7}r^2}{N^2}\right)^{1/2}\\
&=\frac{N^2}{p^8Q|g_{\bar{\chi}}|r^{2/3}}.\frac{p^{3+15/2}r}{N^{3/2}L^{1/2}}\sum_{q'\leq Q/p}q'\\
&=\frac{N^{1/2}Qr^{1/3}}{L^{1/2}}.
\end{aligned}
\end{equation}Finally, substituting $Q=((NL)/p)^{1/2}$, we get the concluding bound for this subsection to be
\begin{equation}
\frac{Nr^{1/3}}{p^{1/2}}
\end{equation}which gets absorbed into the error term of \eqref{eq212} since $L<p$.

\section{Optimal choice for $L$ and $r$}

Combining \eqref{err1}, Lemma \ref{dd}, Lemma \ref{d} and Lemma \ref{ndd}, we conclude 
\begin{equation}
\begin{aligned}
S_r(N)\ll Nr^{1/2}/L^{1/2}+& N^{1/2}pLr^{1/2}+N^{1/2}pL^{1/2}r^{1/2} \\
&+N^{3/4}p^{1/2}L^{3/4}r^{1/2}+N^{3/4}p^{3/4}r^{1/2}/L^{1/4}\,.
\end{aligned}
\end{equation}
Hence
\begin{equation}
\begin{aligned}
\frac{S_r(N)}{N^{1/2}}\ll N^{1/2}r^{1/2}/L^{1/2}+pLr^{1/2}+N^{1/4}p^{1/2}L^{3/4}r^{1/2}+N^{1/4}p^{3/4}r^{1/2}/L^{1/4}\,.
\end{aligned}
\end{equation}

Using the upper bound $N\ll p^{3+\epsilon}/r^2$ and $r\leq p^{\theta}$ we get
\begin{equation}
\begin{aligned}
&L\left(\frac{1}{2},\pi\times f\times\chi\right)\\
&\ll \sup_{r\leq p^{\theta}}\sup_{\frac{p^{3-\theta}}{r^2}\leq N\leq \frac{p^{3+\epsilon}}{r^2}}\frac{|S_r(N)|}{N^{1/2}}+p^{(3-\theta)/2}\,,\\
&\ll p^{\epsilon}(p^{1+\eta+\theta /2}+p^{5/4+3\eta/4}+p^{3/2-\eta/4}+p^{(3-\theta)/2})\,,
\end{aligned}
\end{equation}where we have assumed $L=p^{\eta} $,  for some $0<\eta<1$. Equating the second and the third term, we choose $\eta=1/4$. We then choose $\theta=1/8$, to see that the second and the third term dominates the first and the last and we finally obtain
\begin{equation}\label{94}
L\left(\frac{1}{2},\pi\times f\times\chi\right)\ll  p^{3/2-1/16+\epsilon}.
\end{equation}Note that the choice $\eta=1/4, \theta=1/8$ is consistent with \eqref{eq2.8}.

\section*{Acknowledgements}
The author would like to thank Prof. Ritabrata Munshi for suggesting the
problem, sharing his ideas and explaining his methods. He would also like
to thank Jyoti Sengupta and Philippe Michel for their helpful comments, Will Sawin for his invaluable help and writing the appendix, and Indian Statistical Institute, Kolkata, for the excellent research environment. He is indebted to the anonymous reviewer for careful reading and many useful comments and suggestions. This is a part of the author's Ph.D. thesis written under the supervision of Prof. R. Munshi and Prof. R. Sridharan.

\vspace{20mm}
\newpage
\section*{\textbf{APPENDIX : A CHARACTER SUM BOUND VIA KATZ'S HYPERGEOMETRIC FUNCTIONS}\\ \vspace{5mm} \small{Will Sawin}}

Let $p$ be a prime.  Let $\chi$ be a non-trivial Dirichlet character mod $p$. For $c_1, c_2, k_1, k_2 \in \mathbb Z$, define

\[\mathcal C := \sideset{}{^*}\sum_{\substack{  \alpha(p)  \\ \alpha \neq - k_2/k_1}} \sideset{}{^*}\sum_{x, y(p)} \chi(x+1) \chi(c_1 \overline{x} + \overline{\alpha} ) \overline{\chi} (y+1) \overline{\chi} ( c_2 \overline{y} + \overline{k_1\alpha + k_2}) .\]

The goal of this appendix will be to prove:

\begin{theorem}\label{appendix-main-thm} 
We have $\mathcal C \ll p^{3/2}$ if $k_2 \neq 0 \mod p$.
\end{theorem}

To that end, define a sum $F_{c} (a) $ by
\[ p^{1/2} F_{c}(a) := \sideset{}{^*}\sum_{x(p)} \chi(x+1) \chi ( c \overline{x} +a ) = \sideset{}{^*}\sum_{x(p)} \chi(1 + \overline{x} ) \chi ( c  +ax  ) \]

so that \begin{equation}\label{CF-relation} \mathcal C = p \sideset{}{^*}\sum_{\substack{  \alpha(p)  \\ \alpha \neq - k_2/k_1}}  F_{c_1} ( \overline{\alpha} ) \overline{ F_{ c_2}} ( \overline{k_1 \alpha + k_2 } ).\end{equation}

\begin{lemma}\label{c-zero}
If $c=0$ then $F_c(a) = O (p^{-1/2})$. 
\end{lemma}	

\begin{proof} We have
\[ p^{1/2} F_{c}(a) := \sideset{}{^*}\sum_{x(p)} \chi(x+1) \chi ( a ) = \chi(a) \sideset{}{^*}\sum_{x(p)} \chi(x+1)  = - \chi(a) = O(1) . \qedhere \] \end{proof}

This estimate will suffice to handle the cases when $c_1=0$ or $c_2=0$. The $c \neq 0$ case is more difficult.

\begin{lemma}\label{sheaf-existence} 
If $c \neq 0$, then $F_c$ is the trace function of a sheaf $\mathcal F_c$ on $\mathbb A^1$. The sheaf $\mathcal F_c$ is lisse away from $0, c, \infty $, is pure of weight $0$, is geometrically irreducible, and has conductor (in the sense of \cite[Definition 4.3]{FKMS}) uniformly bounded. 

Furthermore, the local monodromy representations at $0$ and $c$ have nonzero invariants, while the local monodromy representation at $\infty$ does not.
\end{lemma}	
\begin{proof} 

We have $\chi(c+ax) =\chi(c) \chi ( 1 + (a/c) x )$ so $F_c(a) = \chi(c) F_1(a/c)$. 

We thus restrict attention to the case $c=1$. In this case, the substitution $x \mapsto -x$ gives
\[ p^{1/2} F_{1}(a)  = \sideset{}{^*}\sum_{x(p)} \chi(1 + \overline{x} ) \chi ( 1  +ax  ) = \chi(\overline{x} -1) \chi (  ax -1 ) = f * f \] where $f(x) = \chi(x-1)$ and $*$ denotes multiplicative convolution
\[ f* g( a) = \sideset{}{^*}\sum_{x(p)} f( x ) g ( a \overline{x}).\]

We now relate $f$, and thus $F$, to finite field hypergeometric functions defined by Katz \cite{katz-esde}, and hence to trace functions.

Let $\psi \colon \mathbb F_p \to \mathbb C^\times$ be an additive character. Katz defines \cite[(8.2.7)]{katz-esde} a hypergeometric sum $\operatorname{Hyp} (\psi;  \chi\textrm{'s}; \psi\textrm{'s}) (E, t)$, which, in the special case $\operatorname{Hyp} (\psi; 1;\chi) (\mathbb F_p, t)$, is given by
\[ \operatorname{Hyp} (\psi; 1; {\chi}) (\mathbb F_p, t) = \sum_{ \substack{ x,y (p) \mid x=ty  }}  \psi (x-y) \overline{\chi}(y)  = \sum_{ y(p)} \psi ( (t-1) y ) \overline{\chi}(y) = \sum_{z(p)} \psi ( z) \overline{\chi} ( \overline{(t-1) } z) \] \[= \chi(t-1)  \sum_{z(p)} \psi ( z) \overline{\chi} ( z) = f(t) G(\overline{\chi},\psi) \] where $G(\overline{\chi})$ is a Gauss sum.

By the convolution property of hypergeometric functions \cite[(8.2.3)]{katz-esde}, this gives
\[ \operatorname{Hyp} (\psi; 1,1; \chi, {\chi}) (\mathbb F_p , a) = \operatorname{Hyp} (\psi; 1;\chi)* \operatorname{Hyp} (\psi; 1;\chi) (a) = f* f(a) G(\overline{\chi},\psi)^2\] \[=  p^{1/2} F_1(a)   G(\overline{\chi},\psi)^2 .\]

Hence

\[ F_c(a) = \operatorname{Hyp} (\psi; 1,1; \chi, {\chi}) (\mathbb F_p , a/c) \frac{\chi(c) }{  p^{1/2}G(\overline{\chi},\psi)^2}.\]

 By definition \cite[(8.2.7)]{katz-esde}, $ \operatorname{Hyp} (\psi; 1,1; \chi, {\chi}) (\mathbb F_p , a)$ is the trace function of the sheaf Katz calls $\operatorname{Hyp} (!, \psi; 1,1; \chi, {\chi}) $. Thus $ \operatorname{Hyp} (\psi; 1,1; \chi, {\chi}) (\mathbb F_p , a/c)$ is the trace  function of the pullback of that sheaf along multiplication by $1/c$, which Katz calls $\operatorname{Hyp}_{c} (!, \psi; 1,1; \chi, {\chi}) $  \cite[ (8.2.13)]{katz-esde} (he defines it as the pushforward along multiplication by $c$, but this is equivalent as pushforward along an invertible morphism is equivalent to pullback along the inverse map).
 
  Hence $F_c(a)$ is the trace function of the $\alpha^{\operatorname{degree}}$ twist of $\operatorname{Hyp} _{c} (!, \psi; 1,1; \chi, {\chi}) (\mathbb F_p , a)$, where $\alpha =  \frac{\chi(c) } { p^{1/2}  G(\overline{\chi},\psi)^2}$ satisfies $|\alpha| =p^{-3/2}$.
 
By \cite[Theorem 8.4.2(4)]{katz-esde}, $\operatorname{Hyp} _{c} (!, \psi; 1,1; \chi, {\chi})$ is pure of weight $2+2-1=3$, so its $\alpha^{\operatorname{degree}}$ twist is pure of weight $3-3=0$.

The sheaf is geometrically irreducible by \cite[Theorem 8.4.2(1)]{katz-esde}.

The singularities and local monodromy representations are described by  \cite[Theorem 8.4.2(8)]{katz-esde}. In particular, this implies: $\mathcal F_c$ is lisse away from $0,c$. The local monodromy at $c$ is a transvection, meaning the codimension of its invariant subspace is $1$, so the dimension is $2-1=1>0$. The local monodromy at $0$ is the tensor product of a trivial representation with a unipotent representation and thus has nontrivial invariants. The local monodromy representation at $\infty$ is the tensor product of a nontrivial representation of dimension $1$ with a unipotent representation and thus does not have nontrivial invariants.

 Since the rank is $2$, the number of singularities is $3$, and the local monodromy at each singularity is tame, the conductor in the sense of \cite[Definition 4.3]{FKMS} is bounded by $2+3+0=5$. \end{proof}

\begin{proof}[Proof of Theorem \ref{appendix-main-thm}] 

If $c_1=c_2=0$ then by \eqref{CF-relation} and Lemma \ref{c-zero},
\[ \mathcal C = p \sideset{}{^*}\sum_{\substack{  \alpha(p)  \\ \alpha \neq - k_2/k_1}}  O (1/\sqrt{p} ) \cdot O(1/\sqrt{p} ) \ll  p \ll p^{3/2} .\]

When $c \neq 0$, a sheaf $\mathcal F_c$ as in Lemma \ref{sheaf-existence} is pure of weight $0$ and thus satisfies the standing assumption of \cite[Remark 3.11]{FKMS}. A basic property of trace functions of sheaves that are pure of weight $0$ they are $\ll 1$ is that they are bounded by the rank of the sheaf (and thus by the conductor). Thus, when $c_1=0$ but $c_2\neq 0$, or vice versa, we have
\[ \mathcal C = p \sideset{}{^*}\sum_{\substack{  \alpha(p)  \\ \alpha \neq - k_2/k_1}}  O (1/\sqrt{p} ) \cdot O(1 ) \ll  p^{3/2} .\]
This handles the cases when $c_1=0$ or $c_2=0$. For the remainder, we assume $c_1 ,c_2 \neq 0$.

We first handle the case $k_1 \neq 0\mod p$.  In this case, the pullback of $F_c$ along the rational linear transformations $\alpha \mapsto \overline{\alpha}$ and $\alpha \mapsto \overline{k_1 \alpha + k_2 }$ are thus also trace functions (of sheaves obtained by pulling back $\mathcal F_c$ along the same rational linear transformations \cite[\S8]{FKMS}). Since pullback by rational linear transformations preserves geometric irreducibility, these are also trace functions of geometrically irreducible sheaves. (We need $k_1 \neq 0 \mod p$ for $\alpha \mapsto \overline{k_1 \alpha + k_2 }$ to be a rational linear transformation and not a constant function.)

By the Riemann hypothesis in the form \cite[Equation (5.4)]{FKMS}, the sum \[ \frac{\mathcal C}{p} =  \sideset{}{^*}\sum_{\substack{  \alpha(p)  \\ \alpha \neq - k_2/k_1}}  F_{c_1} ( \overline{\alpha} ) \overline{ F_{ c_2} }( \overline{k_1 \alpha + k_2 } )= O( \sqrt{p})\]   unless the sheaves  $[ \alpha \mapsto \alpha^{-1}]^* \mathcal F_{c_1}$ and $[ \alpha \mapsto (k_1\alpha+k_2)^{-1} ]^* \mathcal F_{c_2}$ are geometrically isomorphic. 

If they were geometrically isomorphic, they would have to have the same local monodromy at each point. Since $\mathcal F_{c_1}$ has nontrivial local monodromy invariants at every finite point but not at $\infty$, its pullback $[ \alpha \mapsto \alpha^{-1}]^* \mathcal F_{c_1}$ has nontrivial local monodromy at every point except the inverse image $0$ of $\infty$. 

Similarly, $[ \alpha \mapsto (k_1\alpha+k_2)^{-1} ]^* \mathcal F_{c_2}$ has nontrivial local monodromy invariants at every point except the inverse image $-k_2/k_1$ of $\infty$.

Thus, the sheaves can only be isomorphic if $0 =-k_2/k_1$, which happens only if $k_2=0$ (mod $p$), as desired.

Now in the case $k_1=0$, the sum $\frac{\mathcal C}{p}$ simplifies as \[ \sideset{}{^*}\sum_{\substack{  \alpha(p)  \\ \alpha \neq - k_2/k_1}}  F_{c_1} ( \overline{\alpha} ) \overline{ F_{ c_2} }( \overline{k_2 } )\] and since $\overline{ F_{ c_2} }( \overline{k_2 } )=O(1)$, it suffices to prove that $\sum^*_{\substack{  \alpha(p)  \\ \alpha \neq - k_2/k_1}}  F_{c_1} ( \overline{\alpha} )  = O(\sqrt{p})$ which, since $[ \alpha \mapsto \alpha^{-1} ]^* \mathcal F_{c_1}$ is geometrically irreducible and nontrivial, follows from another application of the Riemann hypothesis \cite[Corollary 4.7]{FKMS} (or from a more elementary estimate by opening up the sum).

 \end{proof}

 I would like to thank Emmanuel Kowalski, Yongxiao Lin, and Philippe Michel for helpful conversations. This work was supported by NSF grant DMS-2101491.

\end{document}